\documentclass[11pt,reqno]{amsart}
\usepackage[utf8]{inputenc}
\usepackage{a4wide}
\usepackage{amsmath,mathrsfs,eucal}
\usepackage{xcolor}
\usepackage{bbm,stix}
\usepackage{enumitem}
\usepackage{graphics} 
\usepackage{epsfig} 
\usepackage{graphicx,subfig}                     
\usepackage{wrapfig}
\usepackage{textcomp}
\usepackage{tikz}
\usepackage{pgfplots}
\usepackage{dsfont}
\pgfplotsset{compat=1.16}
\usepackage{comment}
\usepackage{multicol}  
\usepackage{epstopdf}
\usepackage[colorlinks=true,allcolors=black]{hyperref}
\usepackage{bm}                        \usepackage[]{appendix}  
\usepackage{booktabs}                                                                          
\usepackage{graphicx,subfig}                                                       \usepackage{float} 
\usepackage{wrapfig}              

\newtheorem{defn}{Definition}[section]
\newtheorem{thm}{Theorem}[section]
\newtheorem{prop}{Proposition}[section]
\newtheorem{lem}{Lemma}[section]

\newtheorem{rem}{Remark}[section]

\newcommand{\R}{\mathbb{R}}
\newcommand{\N}{\mathbb{N}}
\newcommand{\Prob}{\mathcal{P}}
\newcommand{\Lip}{\mathrm{Lip}}
\newcommand{\W}{\mathcal{W}}
\newcommand{\K}{\mathcal{K}}
\newcommand{\spt}{\mathop\mathrm{spt}}
\renewcommand{\L}[1]{L^{\pmb #1}}

\def\XXint#1#2#3{{\setbox0=\hbox{$#1{#2#3}{\int}$}
\vcenter{\vspace{-1pt}\hbox{$#2#3$}}\kern-.5\wd0}}
\def\Xint#1{\mathchoice {\XXint\displaystyle\textstyle{#1}}{\XXint\textstyle\scriptstyle{#1}}{\XXint\scriptstyle\scriptscriptstyle{#1}}{\XXint\scriptscriptstyle\scriptscriptstyle{#1}}\!\int}

\def\klintmed{\Xint{\hbox{\vrule height -0pt width 6pt depth 1pt}}}

\title[A conservation law with nonlocal flux]{Measure solutions, smoothing effect, and deterministic particle approximation for a conservation law with nonlocal flux}
\author{M. Di Francesco \and S. Fagioli \and E. Radici}

\begin{document}
\address{ Marco Di Francesco - DISIM - Department of Information Engineering, Computer Science and Mathematics, University of L'Aquila, Via Vetoio 1 (Coppito)
67100 L'Aquila (AQ) - Italy}
\email{marco.difrancesco@univaq.it}

\address{ Simone Fagioli - DISIM - Department of Information Engineering, Computer Science and Mathematics, University of L'Aquila, Via Vetoio 1 (Coppito)
67100 L'Aquila (AQ) - Italy}
\email{simone.fagioli@univaq.it}

\address{ Emanuela Radici - DISIM - Department of Information Engineering, Computer Science and Mathematics, University of L'Aquila, Via Vetoio 1 (Coppito)
67100 L'Aquila (AQ) - Italy}
\email{emanuela.radici@univaq.it}

\begin{abstract}
    We consider a class of nonlocal conservation laws with an interaction kernel supported on the negative real half-line and featuring a decreasing jump at the origin. We provide, for the first time, an existence and uniqueness theory for said model with initial data in the space of probability measures. Our concept of solution allows to sort a lack of uniqueness problem which we exhibit in a specific example. Our approach uses the so-called \emph{quantile}, or \emph{pseudo-inverse} formulation of the PDE, which has been largely used for similar types of nonlocal transport equations in one-space dimension. Partly related to said approach, we then provide a deterministic particle approximation theorem for the equation under consideration, which works for general initial data in the space of probability measures with compact support. As a crucial step in both results, we use that our concept of solution (which we call \emph{dissipative measure solution}) implies an instantaneous \emph{measure-to-$L^\infty$} smoothing effect, a property which is known to be featured as well by local conservation laws with genuinely nonlinear fluxes.
\end{abstract}

\maketitle

\section{Introduction}
This paper deals with the class of nonlocal conservation laws
\begin{equation}\label{eq:main}
    \partial_t \rho + \partial_x (\rho W[\rho(\cdot,t)]) = 0\,,\quad x\in \R\,,\quad t\geq 0\,.
\end{equation}
We provide a concept of solution to \eqref{eq:main} for initial data $\rho_0$ in the space $\Prob(\R)$ of probability measures on $\R$ yielding existence and uniqueness of solutions to the Cauchy problem. 
Here $W$ is a nonlocal operator
\[W:\Prob(\R)\rightarrow L^\infty_{\mathrm{loc}}(\R)\]
defined as follows. We first introduce the \emph{interaction kernel} $V:\R\rightarrow \R$
that we assume to satisfy the following set of standing assumptions:
\begin{itemize}
    \item [(V1)] $\mathrm{supp}(V)\subset(-\infty,0]$;
    \item [(V2)] $V$ is left-continuous at $0$ with $V(0^-)=\lambda>0$.
    \item [(V3)] $V$ is Lipschitz continuous on $(-\infty,0]$.
\end{itemize}
We now introduce the \emph{velocity map} $v:\R\rightarrow \R$ with the assumptions
\begin{itemize}
    \item [(v1)] $v\in \Lip(\R)$;
    \item [(v2)] There exists $b>0$ such that for all $\rho_2\geq\rho_1\geq 0$ one has $v(\rho_2)-v(\rho_1)\leq -b(\rho_2-\rho_1)$.
\end{itemize}
We observe that assumptions (v2) could be relaxed in some cases, see Remark \ref{rem:assumptions} below. We now introduce the nonlocal operator $W$. Given $\rho \in \mathcal{P}(\R)$, we set
\begin{equation}\label{eq:Wh}
    W[\rho](x)=v((V\ast \rho) (x))=v\left(\int_{\R}V(x-y)d\rho(y)\right)\,.
\end{equation}
For simplicity in the notation, in case $\rho$ is absolutely continuous with respect to the Lebesgue measure, we identify the measure with its density, that is, we write $\rho$ instead of $\rho \mathcal{L}^1$.

Nonlocal conservation laws of the form \eqref{eq:main} (along with their various versions and extensions) are gaining more and more attention in the scientific community, mostly due to the fact that they are a very flexible tool to model nonlinear transport. In fact, they are used to  describe various phenomena in natural and social sciences such as crowd movements \cite{pedestrian} and sedimentation \cite{sedimentation}, as well as to model several processes of great interest in engineering and technology such as vehicular traffic \cite{traffic}, supply chains \cite{supply}, and conveyor belts \cite{conveyor}. 

The Cauchy problem for models of the form \eqref{eq:main} with absolutely continuous initial data has been extensively studied in the literature. For the existence and uniqueness theory via finite volume scheme we refer to \cite{goatin_scialanga,chiarello_goatin,chiarello_goatin_rossi}. In this context, initial data are typically taken in $BV(\R)$. Existence in $L^\infty$ via fixed point has been obtained in \cite{keimer_pflug_2017}. A more general result in which the kernel $V$ is $BV$ was obtained in \cite{coclite_denitti_keimer_pflug_2022}. An approach using Wasserstein distances and a deterministic particle approach was addressed in \cite{goatin_rossi}. Very often, especially in traffic flow modelling, the nonlocal operator in \eqref{eq:Wh} is replaced by
\begin{equation}\label{eq:Wh2}
    W_\eta[\rho]=v\left(\int_{x}^{x+\eta}V(x-y)d\rho(y)\right)\,,
\end{equation}
with $V\geq 0$, to emphasise that each vehicle can only interact with vehicles ahead of it and with distance not larger than a fixed $\eta>0$. The kernel $V=V_\eta$ may depend on the proximity parameter $\eta$. In most cases, the supremum of the interaction kernel $V$ blows up like $\eta^{-1}$ for small $\eta>0$, in such a way to make the convolution \[\int_{x}^{x+\eta}V_\eta(x-y)d\rho(y)\] a suitable candidate to approximate $\rho(x)$ in a distributional sense for small $\eta$, and hence the equation \eqref{eq:main} a suitable candidate to approximate the local conservation law 
\begin{equation}\label{eq:local}
    \partial_t \rho + \partial_x (\rho v(\rho))=0
\end{equation}
for small $\eta$. These considerations led some researchers in this field to focus both on the existence theory for \eqref{eq:main} and on its interplay with the local conservation laws when $W[\rho]$ is given by \eqref{eq:Wh2} and as $\eta\searrow 0$. In particular, the latter problem is extremely interesting and challenging and led to very deep results, see for example \cite{CoCrSp2019,bressan_shen_2020,CoCrMaSp2021,CoCrMaSp2023,coclite_colombo_crippa_denitti_keimer_marconi_pflug_spinolo,keimer_pflug_2023_2} (the list is intentionally non-exhaustive as the $\eta\searrow 0$ limit is not a point we address in this paper). 

The use of nonlocality in the transport part eases (apparently) the mathematical theory of the equation $\partial_t \rho + \partial_x (\rho v)=0$ with respect to the more classical approach in which $v$ depends locally on $\rho$ (see for example the classical LWR model for traffic flow \cite{LighthillWhitham}), in that the nonlocal vector field in \eqref{eq:Wh} is continuous in a weaker topology when compared to its local counterpart $v(\rho)$. However, the mathematical theory of \eqref{eq:main} is still challenging due to the possible discontinuity of the kernel $V$. We mention at this stage that a parallel literature was developed in the past 20 years on the so called \emph{nonlocal interaction equation} $\partial_t \rho+ \mathrm{div}(\rho \nabla W \ast \rho)=0$, in which $W$ is typically considered radially symmetric with a possible singularity at the origin, see for example \cite{AGS,CDFLS,carrillo_choi_hauray}. Said equations have applications in various fields including cells biology \cite{keller_segel}, models for swarms \cite{topaz,mogilner}, collective motion \cite{CCR}, granular media flow \cite{BenedettoCagliotiPulvirenti97}, and opinion formation \cite{sznajd}. We refer to \cite{AGS,CDFLS,carrillo_choi_hauray} as the main references for existence and uniqueness results in a measure sense (the list is far for being exhaustive). 

In one space dimension (see e.g. \cite{Burger_DF_NHM,Bonaschi_Carrillo_DF_Peletier} and the recent \cite{DF_Schmidtchen_Iorio}), nonlocal interaction equations may be expressed in a very convenient way using the pseudo-inverse representation of the cumulative distribution function (see Subsection \ref{subsec:wasserstein} below), which turns out to be very useful for \eqref{eq:main} as well. One of the starting motivations of the present manuscript is actually that of expressing \eqref{eq:main} in said pseudo-inverse (or quantile) formulation. We briefly anticipate here the key aspects of this approach, which 
dates back to \cite{carrillo_toscani_wasserstein,carrillo_gualdani_toscani}, and refer to Subsections \ref{subsec:wasserstein}, \ref{subsec:continuity}, \ref{subsec:non_uniqueness}, and \ref{subsec:new} below for further details. For a given probability measure $\rho\in \mathcal{P}(\R)$, we define its cumulative distribution $F_\rho:\R\rightarrow \R$
\[F_\rho(x)=\rho((-\infty,x])\]
and the pseudo-inverse (or quantile function) $X_\rho:[0,1]\rightarrow \R$
\begin{equation*}
  X_\rho(z)=\inf\left\{x\in \R\,:\,\,F_\rho(x)\geq z\right\}\,.  
\end{equation*}
Assuming $\rho(\cdot,t)$ solves \eqref{eq:main} in a classical sense, we denote by $F_\rho(\cdot,t)$ the cumulative distribution of $\rho(\cdot,t)$ and by $X_\rho(\cdot,t)$ the pseudo-inverse of $F_\rho(\cdot,t)$. We assume for simplicity that $F_\rho$ is strictly increasing on $\R$ for fixed times and has no jumps, so that $X_\rho(\cdot,t)$ is the actual inverse of $F_\rho(\cdot,t)$. We also assume enough regularity on $\rho$, $F_\rho$, $X_\rho$ and on the given functions $v$ and $V$. We then formally recover, after a few computations involving chain rules, the following equation satisfied by $X_\rho(z,t)$, 
\begin{equation}\label{eq:quantile_intro0}
    \partial_t X_\rho(z,t)=v\left(\int_0^1 (V(X_\rho(z,t)-X_\rho(\zeta,t))d\zeta \right)\,,\qquad (z,t)\in [0,1]\times [0,+\infty)\,.
\end{equation}
We remark at this stage that $X_\rho$ is strictly increasing if $\rho$ has no atoms, whereas $X_\rho$ is constant on a sub-interval of $[0,1]$ in case $\rho$ has an atomic part. For further details we refer to Subsection \ref{subsec:wasserstein}. The rigorous correspondence between \eqref{eq:main} and \eqref{eq:quantile_intro0} is discussed in Subsections \ref{subsec:continuity} and \ref{subsec:non_uniqueness}. To bypass the discontinuity of the kernel $V$ at the origin, we will also consider the following alternative formulation for \eqref{eq:quantile_intro0}, the derivation of which is contained in Subsection \ref{subsec:non_uniqueness} below,
\begin{equation}\label{eq:quantile_intro}
    \partial_t X_\rho(z,t)=v\left(\int_0^1 (U(X_\rho(z,t)-X_\rho(\zeta,t))d\zeta - \lambda z\right)
\end{equation}
where $U:\R\rightarrow \R$ is defined by
\[U(x)=\begin{cases}
    V(x) & x\leq 0\\
    \lambda & x>0
\end{cases}\]
with $\lambda=V(0^-)$ as in the assumption (V2).

The equation \eqref{eq:quantile_intro0} may be seen as an integro-differential equation in a suitable functional framework for the $X_\rho$ variable, in which $X_\rho$ is assumed to be non-decreasing with respect to $z$. The typical choice for the functional space is $L^p([0,1])$ for $p\in [1,+\infty]$. In particular, requiring $X_\rho(\cdot,t)\in L^p([0,1])$ for finite $p\geq 1$ corresponds to requiring $\rho(\cdot,t)$ to have finite moment of order $p$. Clearly, as long as the velocity map $v$ and the kernel $V$ are assumed to be globally Lipschitz on the whole $\R$, existence and uniqueness for \eqref{eq:quantile_intro0} easily follow from a suitable $L^p([0,1])$-version of Cauchy-Lipschitz-Picard theorem, as the operator on the right-hand side of \eqref{eq:quantile_intro0} is in that case a Lipschitz operator on $L^p([0,1])$. Such an approach cannot be applied directly in case $V$ has a discontinuity as in our case due to (V1)-(V2). Our alternative formulation \eqref{eq:quantile_intro} allows us to bypass this problem in our setting. It is worth mentioning at this stage that the case of Lipschitz kernel $V$ is covered in the theory developed in \cite{crippa_lecureux} in a multidimensional framework (thus, not based on the pseudo-inverse approach).

A first property of \eqref{eq:main} we will detect in this paper is that a jump discontinuity of the kernel $V$ combined with an initial datum $\rho_0$ with a \emph{nontrivial atomic part} results into a possible \emph{lack of uniqueness} of solutions to the Cauchy problem for \eqref{eq:main} in a weak measure sense. Indeed, a simple computation in the case $v(\rho)=1-\rho$, $V(x)=\mathbf{1}_{(-\infty,0]}(x)$ (see Subsection \ref{subsec:non_uniqueness} below) proves that if $\rho$ satisfies \eqref{eq:main}, then the cumulative distribution $F=F_\rho$ formally satisfies the well-known Burgers equation
\[\partial_t F + \partial_x (F^2/2)=0\,.\]
Hence, the initial datum $\rho_0=\delta_0$, corresponding to the initial discontinuous cumulative distribution  $F(x,0)=\mathbf{1}_{[0,+\infty)}(x)$, implies the existence of multiple solutions in the $F$-landscape. As is well known in the classical theory of conservation laws (see \cite{oleinik,kruzkov}), the unique entropy solution $F(x,t)$ is, in this case, the one in which the initial discontinuity at $x=0$ dissolves into a rarefaction wave, which in terms of $\rho=F_x$ means that the initial Delta singularity in $\rho_0$ immediately disappears and the solution $\rho(\cdot,t)$ becomes $L^\infty$ instantaneously for $t>0$. We stress at this stage that no lack of uniqueness is featured by \eqref{eq:main} in the case of initial conditions which are absolutely continuous with respect to the Lebesgue measure. This non uniqueness may only occur if we allow for concentrated initial data.

Based on the above considerations, one of our contribution to this theory will be setting up a well-posedness theory for  \eqref{eq:main} with initial datum  $\rho_0$ in the space of probability measures on $\R$ automatically selecting the unique solution featuring the above described \emph{measure-to-$L^\infty$} smoothing effect. A similar issue was tackled in \cite{Bonaschi_Carrillo_DF_Peletier} for linear $v$ and $V(x)=\pm\mathrm{sign}(x)$. We will refer to said solution as  \emph{dissipative measure solution}. We will then rely on \eqref{eq:quantile_intro0} and its modified version \eqref{eq:quantile_intro} to formulate our existence and uniqueness result
and prove that if we start with an arbitrary non-increasing initial condition for $X_\rho$ in \eqref{eq:quantile_intro} (corresponding to a measure initial condition for $\rho$), $X_\rho(\cdot,t)$ becomes strictly increasing for positive times, with a $z$-slope which is controlled from below by a (time depending) positive constant. This will imply that the corresponding measure $\rho(\cdot,t)$ solving \eqref{eq:main} is absolutely continuous with respect to the Lebesgue measure for all positive times. All these arguments together lead to prove existence, uniqueness, and stability for \eqref{eq:main} in a probability measure framework. 

The desired instantaneous measure-to-$L^\infty$ smoothing effect for \eqref{eq:main} will be included in our main result. The interest towards this property is that, despite the nonlocal model \eqref{eq:main} and the local one \eqref{eq:local} are quite different, the \emph{local} conservation law \eqref{eq:local} satisfies somehow a similar smoothing effect, see \cite{liu_pierre}. Apart from that, this property may be also interesting from the applications point of view, for example in vehicular traffic, to model situations in which initially very large concentrations of vehicles are subject to instantaneous (singular)  \emph{repulsive} effects, forcing consecutive vehicles to adjust to a 
distance of the same order as the mass of a single vehicle. We stress that the discontinuity of the kernel $V$ at the origin plays a crucial role to obtain such an effect. 

As is clear in traffic flow models, the interpretation of continuum models from the point of view of \emph{moving particles} is crucial to understanding the intuition behind these models at a microscopic scale, as well as to validate the use of continuum models for models which are intrinsically discrete in their formulation. Hence, another goal of this paper is to formulate a \emph{deterministic particle approximation} for \eqref{eq:main}. A similar task has been performed in \cite{goatin_rossi} with bounded initial data. Our goal is to do so for measure initial conditions, with the hope of catching, at a discrete particle level, the same smoothing effect emphasised above. The discrete particle counterpart of \eqref{eq:main} we consider here is the ODE system
\begin{equation}\label{eq:particle_intro}
   \dot{x}_i(t)= v\left(m_N \sum_{k=i+1}^N V(x_i(t)-x_k(t))\right)\,,\qquad i=0,\ldots,N\,, 
\end{equation}
with $m_N=1/N$ and with the convention $V(0)=V(0^-)=\lambda$. We notice that in \eqref{eq:particle_intro} each particle $x_i$ interacts with all particles $x_k$ with index $k>i$. This set consists of all particles $x_k$ with position strictly larger than $x_i$ plus the particles $x_k$ with same position as $x_i$ and with index $k>i$. Such a choice in our scheme may seem artificial but is indeed quite natural given similar results in \cite{Bonaschi_Carrillo_DF_Peletier,DF_Schmidtchen_Iorio} and is somehow inspired by the quantile formulation \eqref{eq:quantile_intro0}.

In case $V$ is smooth on the whole $\R$, assuming the particles trajectories $x_i(t)$ satisfy \eqref{eq:particle_intro}, proving that
\begin{equation}\label{eq:empirical}
    \rho^N(t)=m_N\sum_{i=1}^N \delta_{x_i(t)}
\end{equation}
solves \eqref{eq:main} in a distributional sense is an easy exercise. The presence of a discontinuity in $V$ has been considered in the case $v(\rho)=-\rho$ and $V(x)=\pm\mathrm{sign}(x)$ in \cite{Bonaschi_Carrillo_DF_Peletier}, see also \cite{DF_Schmidtchen_Iorio} for the case $V(x)=-\mathrm{sign}(x) e^{-|x|}$. The use of a deterministic particle approximation to solve a local scalar conservation law in the entropy sense \eqref{eq:local} was first introduced in \cite{DF_rosini,DF_fagioli_rosini}, see also \cite{holden1,holden2}. We stress that the first use of deterministic particles to solve a one-dimensional transport PDE dates back to \cite{russo} for the heat equation, see also \cite{gosse} for the porous medium equation. In many of the aforementioned cases, the reconstruction of the measure \eqref{eq:empirical} from the moving particles $x_1,\ldots,x_N$ is no use to pass the particle scheme to the $N\rightarrow +\infty$ limit. Indeed, the best compactness one can get on $\rho^N$ is in a weak measure topology, whereas at least weak $L^1$ compactness is often needed in the models under study (or even strong $L^1$ compactness, as is the case, for example, for the local conservation law \eqref{eq:local}). Hence, we shall instead consider the discrete formulation for the density
\begin{equation}\label{eq:piecewise}
    \rho^N(x,t)=\sum_{i=0}^{N-1} \mathbf{1}_{[x_i(t),x_{i+1}(t))}(x)\frac{m_N}{x_{i+1}(t)-x_i(t)}
\end{equation}
which is constructed out of $N+1$ moving particles $x_0(t),\ldots,x_N(t)$, see Section \ref{sec:DPA} below.


\medskip
We summarise below the content of the paper and the main results. 
\begin{itemize}
    \item In Section \ref{sec:preliminaries} we first provide preliminaries on $p$-Wasserstein spaces (Subsection \ref{subsec:wasserstein}) and we recall basic concepts on the one-dimensional continuity equation (Subsection \ref{subsec:continuity}). Then, in Subsection \ref{subsec:non_uniqueness} we provide an example in which multiple solutions may arise for the Cauchy problem of \eqref{eq:main}, which justifies the need of our definition of dissipative measure solutions in Subsection \ref{subsec:dissipative}, see Definition \ref{def:dissipative} (after having introduced the quantile re-formulation of the model in Subsection \ref{subsec:new}). We state our main existence and uniqueness result of dissipative measure solutions with initial data in the $p$-Wasserstein space in Theorem \ref{thm:main}. In the same Theorem we also state the main measure-to-$L^\infty$ smoothing property \eqref{eq:smoothing}. In Subsection \ref{subsec:DPA} we introduce our deterministic particle approximation and state our many particle limit result in Theorem \ref{thm:DPA}.
    \item In Section \ref{sec:proof_main} we prove Theorem \ref{thm:main}. The measure-to-$L^\infty$ smoothing property is proven in particular in Propositions \ref{prop:X_increasing} and \ref{prop:Linfty}. Separately from the proof of the main Theorem, we provide improved versions of the $L^\infty$ bound for the solution under more refined assumptions on the kernel $V$ and/or the initial condition, see Propositions \ref{prop:Linfinity_bound} and \ref{prop:Linfinity_bound2}.
    \item In section \ref{sec:DPA} we prove Theorem \ref{thm:DPA}. The discrete counterpart of the measure-to-$L^\infty$ smoothing effect is contained in Lemma \ref{l:reg_eff}. We also provide a global-in-time $L^\infty$ bound on the discrete density in Proposition \ref{prop:maximum}, the proof of which mimics the main idea behind Proposition \ref{prop:Linfinity_bound}.
\end{itemize}

\section{Preliminaries and statements of the main results}\label{sec:preliminaries}

\subsection{The $p$-Wasserstein space of probability measures in one space dimension}\label{subsec:wasserstein}
For a probability measure $\rho\in \Prob(\R)$, we define the cumulative distribution function
\[F_\rho(x)=\rho((-\infty,x])\]
and the quantile (or pseudo-inverse) function $X_\rho:[0,1]\rightarrow \R$
\[X_\rho(z)=\inf\left\{x\in \R\,:\,\, F_\rho(x)\geq z\right\}\,.\]
Notice that for every $\rho\in \Prob(\R)$, the function $F_\rho:\R\rightarrow [0,1]$ is monotone non-decreasing and right-continuous by definition. Moreover, $X_\rho:[0,1]\rightarrow \R$ is monotone non-decreasing too and therefore has (locally) finite total variation, which implies it has a right-continuous representative. We recall that
\[\rho=(X_\rho)_{\#}\left(\mathcal{L}^1_{[0,1]}\right)\,,\]
where $\mathcal{L}^1_{[0,1]}$ is the one-dimensional Lebesgue measure on $[0,1]$. The above is referred to $\rho$ being the \emph{push-forward} measure of $\mathcal{L}^1_{[0,1]}$ through the map $X_\rho$, which also reads
\begin{equation}\label{eq:change_of_variable}
    \int_\R\varphi(x)d\rho(x) = \int_0^1\varphi(X_\rho(z))dz\,,
\end{equation}
for all $\rho$-integrable, Borel measurable functions $\varphi$, see for example \cite{AGS,villani,santambrogio_book}. 

We denote by $\Prob_p(\R)$ the space of probability measures on $\R$ with finite $p$-moment, with $p\in [1,+\infty]$. Here, the $p$-moment of $\rho\in \Prob(\R)$ is defined as
\[M_p[\rho]=\int|x|^p d\rho(x)\]
for $p$ finite, and
\[M_\infty[\rho]=\mathrm{ess sup}\left\{|x|\,:\,\, \hbox{$x\in \mathrm{supp}(\rho)$}\right\}\,.\]
It is well known that $\Prob_p(\R)$ is a complete metric space equipped with the $p$-Wasserstein distance $\W_p$, the definition of which is omitted (see \cite{villani,AGS}) and provided only in the one-dimensional case in \eqref{eq:isometry} below. Indeed, it is also well-known that the mapping $\Prob_p(\R)\ni \rho \mapsto X_\rho \in L^p([0,1])$ is one-to-one and isometric, more precisely
\begin{equation}\label{eq:isometry}
    \W_p(\rho_1,\rho_2)=\|X_{\rho_1}-X_{\rho_2}\|_{L^p([0,1])}\,.
\end{equation}
Since $X_\rho$ is monotone non-decreasing for all $\rho\in\Prob(\R)$, by defining the closed convex cone in $L^p([0,1])$
\[\K_p=\left\{X\in L^p([0,1])\,:\,\, \hbox{$X$ is non decreasing}\right\}\,,\]
we have that the map $\Prob_p(\R)\ni \rho \mapsto X_\rho \in \K_p$ is an isometric bijection. For future reference, we denote the map 
\begin{equation}\label{eq:T_map}
    T:\Prob_p(\R)\rightarrow \K_p \qquad \hbox{such that $T(\rho)=X_\rho$.}
\end{equation}


\begin{rem}[Pseudo-inverses vs. inverses]\label{rem:considerations}
    \emph{If $F_\rho$ is continuous and strictly increasing as a function from the convex hull of the support of $\rho$, $\mathrm{Co}(\mathrm{supp}(\rho))$, to the interval $[0,1]$, then $X_\rho:[0,1]\rightarrow \mathrm{Co}(\mathrm{supp}(\rho))$ coincides with the inverse of $F$ restricted to $\mathrm{Co}(\mathrm{supp}(\rho))$. Therefore, in this case we have the obvious relations $x=X_\rho(F_\rho(x))$ for all $x\in \mathrm{Co}(\mathrm{supp}(\rho))$ and $z=F_\rho(X_\rho(z))$ for all $z\in [0,1]$. This situation corresponds to the case in which $\rho$ is atom-less and has a connected support. The right-continuity of both $F_\rho$ and $X_\rho$ implies that the relation $z=F_\rho(X_\rho(z))$ is still true for all $z\in [0,1]$ in case $\rho$ is atom-less without restriction on the support of $\rho$. In this case the identity $x\leq X_\rho(F_\rho(x))$ holds. In case $\rho$ has atoms and has connected support, in general we have $z\leq F_\rho(X_\rho(z))$ and $x=X_\rho(F_\rho(x))$.} 
\end{rem}

\subsection{CDF and quantile formulations of the $1$d continuity equation}\label{subsec:continuity}

We now consider the one-dimensional continuity equation
\begin{equation}\label{eq:CE}
       \partial_t \rho + \partial_x (\rho G) = 0 
    \end{equation}
with a given measure initial condition
\begin{equation}\label{eq:CE_initial}
    \rho(\cdot,0)=\rho_0\in \Prob_p(\R)\,,
\end{equation}
for some $p\in [1,+\infty]$. We assume
\begin{align}
    & \hbox{$G\in C(\R\times [0,+\infty))$ and $G$ is Lipschitz continuous in $x$ uniformly in $t$}\,.\label{eq:assumptions_G}
\end{align}
As a consequence of \eqref{eq:assumptions_G}, $G$ is locally bounded. For a fixed $\overline{X}\in \K_p$ right-continuous, we consider the Cauchy problem in $L^p([0,1])$
\begin{equation}\label{eq:cauchy_continuity_1}
    \begin{cases}
        \partial_t X(z,t) = G(X(z,t),t) & \\
        X(z,0)=\overline{X}(z)\,. &
    \end{cases}
\end{equation}
We have the following theorem, the proof of which can be found e.g. in \cite[Theorem 2.1]{DF_Schmidtchen_Iorio} (see also \cite[Theorem 4.4]{santambrogio_book}) in case $G$ is globally bounded. The case of $G$ locally bounded easily follows from \cite[Lemma 8.1.4, Lemma 8.1.6, Proposition 8.1.8]{AGS}.

\begin{thm}\label{thm:continuity}
    Assume $G$ satisfies \eqref{eq:assumptions_G}. Assume that $\rho\in C([0,+\infty);\,\Prob_p(\R))$ is a weak measure solution to \eqref{eq:CE} with $\rho_0$ as initial condition. Then, given $X(z,t)=(T\rho(\cdot,t))(z)$ with $T$ being the isometric map \eqref{eq:T_map}, $X$ satisfies the Cauchy problem \eqref{eq:cauchy_continuity_1} with $\overline{X}=T_{\rho_0}$. Vice versa, assume $X\in L^\infty([0,T]\,;\mathcal{K}_p)$ satisfies \eqref{eq:cauchy_continuity_1} almost everywhere on $(z,t)\in [0,1]\times [0,+\infty)$ for some $\overline{X}\in \mathcal{K}_p$. 
    Let $\rho(x,t)=T^{-1}(X(\cdot,t))(x)$ and assume further that $\rho(\cdot,t)\in L^\infty(\R)$ for all $t>0$. Then $\rho$ is a weak measure solution to the continuity equation \eqref{eq:CE}
    with initial condition $\rho(x,0)=\rho_0=T^{-1}(\overline{X})$. 
\end{thm}


We shall refer to the differential equation in \eqref{eq:cauchy_continuity_1} as the \emph{quantile reformulation} of the continuity equation \eqref{eq:CE}.
We now consider the Cauchy problem
\begin{equation}\label{eq:cauchy_primitive}
\begin{cases}
    \partial_t F(x,t)+ G(x,t)\partial_x F(x,t)=0 & \\
    F(x,0)=\overline{F}(x)\,,
\end{cases}
\end{equation}
where $\overline{F}=F_{\rho_0}$ for some $\rho_0\in \Prob_p(\R)$. Formally, the PDE in \eqref{eq:cauchy_primitive} is obtained by integrating in $x$ the continuity equation \eqref{eq:CE} with $F_x=\rho$.

\begin{thm}\label{thm:primitive}
     Assume $G$ satisfies \eqref{eq:assumptions_G}. Assume there exists a weak solution $F(x,t)$ to \eqref{eq:cauchy_primitive} with $F(\cdot,t)$ non-decreasing, right-continuous in $x$ and such that $\overline{F(\R,t)}=[0,1]$. Then, the distributional derivative $\rho(\cdot,t)=\partial_x F(\cdot,t)$ satisfies the continuity equation \eqref{eq:CE} in a weak measure sense with initial datum $\rho_0$.
\end{thm}
For the proof of this Theorem, see the Appendix \ref{app} at the end of this manuscript.
We shall refer to the equation in \eqref{eq:cauchy_primitive} as the \emph{CDF reformulation} of the continuity equation \eqref{eq:CE}.

\subsection{The nonlocal velocity field in the CDF reformulation: a non-uniqueness issue}\label{subsec:non_uniqueness}
Let us now analyse more closely the nonlocal equation \eqref{eq:main}. Our goal here is to provide its CDF reformulation. To do so, we have to analyse the velocity field
\begin{equation}\label{eq:G}
  G(x,t)=v(V\ast\rho(x,t))\,.  
\end{equation}
Assuming (V1)-(V2)-(V3) above, by setting
\begin{align*}
& U(x)=\begin{cases}
V(x) & x\leq 0\\
\lambda & x >0
\end{cases}\\
& H(x)=\begin{cases}
0 & x\leq 0\\
1 & x >0
\end{cases}
\end{align*}
we obtain
\begin{equation}\label{eq:V_dec}
    V(x)=U(x)-\lambda H(x)\,.
\end{equation}
Notice that $U$ is a Lipschitz continuous function on $\R$ with $[U]_{\mathrm{Lip}}=[V]_{\mathrm{Lip}((-\infty,0])}$.

In order to write the nonlocal velocity field $G$ in the quantile and CDF reformulations, we assume for simplicity $\rho\in L^\infty(\R\times (0,+\infty))$. We compute
\begin{align*}
 G(x,t)& \ =v\left(\int_{\R}U(x-y)\rho(y,t)dy - \lambda \int_{\R} H(x-y)\rho(y,t)dy\right) \\
& \ =  v\left(\int_{\R}U(x-y)\rho(y,t) dy- \lambda \int_{-\infty}^x \rho(y,t)dy\right) \\
& \ = v\left(\int_{\R}U(x-y)d\rho(y,t) - \lambda F_\rho(x,t)\right) \,.
\end{align*}
The formal assumption $\rho\in L^\infty(\R\times (0,+\infty))$ and our basic assumptions on $v$ and $U$ imply the above $G$ is Lipschitz continuous. Hence (at least formally), with the notation $F=F_\rho$, the CDF reformulation of \eqref{eq:main} is
\begin{equation}\label{eq:CDF}
    \partial_t F(x,t) +  v\left(\int_{\R}U(x-y)d F_x (y,t) - \lambda F(x,t)\right) \partial_x F(x,t)= 0\,.
\end{equation}
The reference space for $F$ is the set of right-continuous, non-decreasing functions $F:\R\rightarrow [0,1]$. 
Let us now consider the prototype example 
\[v(\rho)=1-\rho\,,\qquad V(x)=\mathbf{1}_{(-\infty,0]}(x)\,.\]
In that case, we have $U\equiv 1$ and $\lambda=1$, which implies
\[
v\left(\int_{\R}U(x-y)d F_x (y,t) - \lambda F(x,t)\right) = 1-\left(\int_{\R}d F_x(y,t) - F(x,t)\right) = F(x,t)\,.
\]
Consequently, $F$ formally satisfies the Burgers equation
\begin{equation}\label{eq:burgers}
    \partial_t F + \partial_x (F^2/2) = 0\,.
\end{equation}
In order to understand the behavior of \eqref{eq:main} in this particular case with an initial condition given by a singular measure, let us fix $\rho_0=\delta_0$ as a Dirac delta measure initial condition for \eqref{eq:main}. The corresponding initial condition for the CDF formulation \eqref{eq:burgers} is 
\[F_0(x)=\begin{cases}
0 & \hbox{if $x<0$}\\
1 & \hbox{if $x\geq0$}\,.
\end{cases}\]
It is well known that two possible weak solutions to \eqref{eq:burgers} arise in this case, namely the \emph{rarefaction wave}
\[F_1(x,t)=
\begin{cases}
    0 & \hbox{if $x<0$}\\
    x/t & \hbox{if $0\leq x< t$}\\
    1 & \hbox{if $x\geq t$}
\end{cases}
\]
and the \emph{shock wave}
\[
F_2(x,t)=\begin{cases}
    0 & \hbox{if $x<t/2$}\\
    1 & \hbox{if $x\geq t/2$}\,.
\end{cases}
\]
The former, $F_1$, corresponds in the $\rho$ variable to a \emph{diffused bump}
\[\rho_1(x,t)=\frac{1}{t}\mathbf{1}_{[0,t]}(x)\,,\]
whereas the latter, $F_2$, corresponds in $\rho$ landscape to a \emph{moving particle}
\[\rho_2(\cdot,t)=\delta_{t/2}\,.\]
It is also well known that the concept of \emph{entropy solution} formulated by Oleinik \cite{oleinik} and Kruzkov \cite{kruzkov} allows to single out a unique solution, which in this particular case is $F_1$.

\subsection{The quantile reformulation of \eqref{eq:main}}\label{subsec:new}
We shall use the above example as a guiding paradigm to derive our concept of solutions and to obtain the quantile reformulation of \eqref{eq:main}. The example in subsection \ref{subsec:non_uniqueness} shows that, in view of the decreasing discontinuity of $V$ at the origin, we may have multiple measure solutions to \eqref{eq:main} with a fixed initial datum in the space of probability measures. The example suggests that the "correct" solution is the one in which the initially concentrated mass instantaneously regularises to become $L^\infty$. A possible way to formulate a suitable theory could be to work directly in the CDF-landscape also for more general $v$ and $U$, thus using a concept of entropy solution for \eqref{eq:CDF} (which, however, features some additional nonlocality in the general case). However, we shall not perform this task in this paper. Instead, we shall rely on the quantile reformulation to provide a proper concept of solution. Such a concept of solution should encompass the smoothing effect featured by the $\rho_1$ solution in the above example. 

Let us (formally) recover the quantile reformulation for $X_\rho=X(z,t)$, $z\in [0,1]$, $t\geq 0$. We need to write $G(X(z,t),t)$ with $G$ given by \eqref{eq:G}. More precisely, we need to compute
\begin{align*}
& v\left(\int_\R V(x-y)d\rho(y,t)\right) = v\left(\int_{x}^{+\infty}V(x-y)d\rho(y,t)\right)\\
& \ = v\left(\int_{\R}U(x-y)d\rho(y,t) - \lambda \int_{-\infty}^x d\rho(y,t)\right)\\
    & \ = v\left(\int_\R U(x-y)d\rho(y,t)-\lambda F(x,t)\right)\qquad \hbox{on $x=X(z,t)$}\,.
\end{align*}
Due to \eqref{eq:change_of_variable}, we can write the convolution term as
\[\int_\R U(x-y)d\rho(y,t) = \int_0^1 U(X(z,t)-X(\zeta,t))d\zeta\,.\]
Now, in order to express the term $F(X(z,t),t)$, guided by the previous example we deliberately choose that $\rho=F_x$ is bounded (and hence atom-less), with the idea that this is what we expect to happen for positive times due to the instantaneous smoothing effect. From Remark \ref{rem:considerations}, we then have $F(X(z,t),t)=z$. This leads to the following quantile reformulation for \eqref{eq:main}:
\begin{equation}\label{eq:pseudo1}
     \partial_t X(z,t)= v\left(\int_0^1 U(X(z,t)-X(\zeta,t)) d\zeta -\lambda z\right)
\end{equation}
coupled with the initial condition
\begin{equation}\label{eq:pseudo1_initial}
    X(0,t)=\overline{X}(z)\,.
\end{equation}

We stress the importance of having used $\rho\in L^\infty$ to obtain \eqref{eq:pseudo1}, in particular to detect the identity $F(X(z,t),t)=z$. Consider the solution $F_2$ of the previous example in the CDF framework. In that case $\rho_2=(F_2)_x$ is a moving Dirac delta, and the $\rho\in L^\infty$ assumption is far from being satisfied. Indeed, in that case the identity $F_2(X_{\rho_2}(z,t),t)=z$ is false and one cannot recover the quantile reformulation \eqref{eq:pseudo1}. 

\subsection{Dissipative measure solutions: statement of the main result}\label{subsec:dissipative}

Based on the above considerations, we are now ready to provide our concept of solution to \eqref{eq:main}.

\begin{defn}[Dissipative measure solutions]\label{def:dissipative}
    Let $\rho_0\in \Prob_p(\R)$ for some $p\in [1,+\infty]$. A function $\rho\in C([0,+\infty);\Prob_p(\R))$ is a \emph{dissipative measure solution} to \eqref{eq:main} if
    \begin{itemize}
        \item [(i)] $\rho(\cdot,t)\in L^\infty(\R)$ for all $t>0$.
        \item [(ii)] $\rho$ satisfies \eqref{eq:main} in a distributional sense on $(x,t)\in \R\times (0,+\infty)$, that is,
\[
\int_{0}^T\int_{\R}\rho(x,t)\partial_t\varphi(x,t)\,dx\,dt+\int_{0}^T\int_{\R}\rho(x,t)v(V\ast\rho)(x,t)\partial_x\varphi(x,t)\,dx\,dt=0\,,
\]
for all $\varphi\in C^1_c(\R\times (0,+\infty))$.
        \item [(iii)] $\rho(\cdot,0)=\rho_0$.
    \end{itemize}
\end{defn}

We may now state our existence and uniqueness result.

\begin{thm}[Existence and uniqueness of dissipative measure solutions]\label{thm:main}
    Let $\rho_0\in \Prob_p(\R)$ for some $p\in [1,+\infty]$. Then, there exists one and only one dissipative measure solution $\rho$ to \eqref{eq:main} in the sense of Definition \ref{def:dissipative}. Moreover, $\rho$ satisfies the following additional properties.
    \begin{itemize}
        \item [(i)]
       Given $X(\cdot,t)=T_{\rho(\cdot,t)}$ with the isometric map $T$ defined in \eqref{eq:T_map}, $X(\cdot,0)=T(\rho_0)$ and the function $X$ satisfies
     \begin{equation}\label{eq:pseudo_main}
      \partial_t X(z,t)= v\left(\int_0^1 U(X(z,t)-X(\zeta,t)) d\zeta -\lambda z\right)
 \end{equation}
 almost everywhere on $(z,t)\in [0,1]\times (0,+\infty)$.
        \item [(ii)] $\rho\in L^\infty_{\mathrm{loc}}((0,+\infty);L^\infty(\R))$ and the following inequality holds
        \begin{equation}\label{eq:smoothing}
            \|\rho(\cdot,t)\|_{L^\infty(\R)}\leq \frac{[v]_{\mathrm{Lip}}[U]_{\mathrm{Lip}}}{b\lambda}\left(1-e^{-[v]_{\mathrm{Lip}}[U]_{\mathrm{Lip}}t}\right)^{-1}
        \end{equation}
        for all $t>0$.
    \end{itemize}
Moreover, given two dissipative measures solutions $\rho_1$ and $\rho_2$ with initial data $\rho_{1,0}$ and $\rho_{2,0}$ respectively, $\rho_{0,1},\rho_{0,2}\in \Prob_p(\R)$ for some $p\in [1,+\infty]$, the following stability estimate holds:
\begin{equation}\label{eq:stability}
    \mathcal{W}_p(\rho_{1}(\cdot,t),\rho_2(\cdot,t))\leq e^{Ct}\mathcal{W}_p(\rho_{0,1},\rho_{0,2})
\end{equation}
with $C$ given by \eqref{eq:lipschitz_constant} below.
\end{thm}
As we shall see in the proof of Theorem \ref{thm:main} carried out in Section \ref{sec:proof_main}, a major point in it will be the existence and uniqueness of solutions for \eqref{eq:pseudo_main} in $L^p([0,1])$.

\begin{rem}
\emph{
While introducing our concept of solution in Definition \ref{def:dissipative}, the natural question arises on why should we use the concept of entropy solution in the $F$-landscape as a criterion to derive a proper concept of solution to \eqref{eq:main}. The answer to this question is that, as is well known in the context of local conservation laws \eqref{eq:local}, entropy solutions are \emph{stable}. Indeed, such a property somehow reflects in the stability property \eqref{eq:stability} stated in Theorem \ref{thm:main}. We observe in particular that the case $p=1$ in \eqref{eq:stability} implies stability in $L^1$ in the $F$-landscape.}
\end{rem}

\begin{rem}[Assumption (v2)]\label{rem:assumptions}
\emph{
It is worth remarking at this stage that 
assumption (v2) would not be satisfied by functions $v$ being supported on an interval $[0,\rho_{\max}]$ and strictly monotone decreasing therein. However, we will show in Proposition \ref{prop:Linfinity_bound} below that if $V$ is nonnegative and non-decreasing on $(-\infty,0]$ then an $L^\infty$ initial condition $0\leq \rho_0\leq \rho_{\max}$ yields a solution $\rho(\cdot,t)\in [0,\rho_{\max}]$. Hence, the strict monotonicity on $v$ in (v2) only needs to be required on the interval $[0,\rho_{\max}]$ in this case.
}
\end{rem}

\subsection{Deterministic particle approximation: a convergence result}\label{subsec:DPA}
The second task of this paper is to provide a deterministic particle approximation (or Follow-the-Leader approximation) to the solution provided in Theorem \ref{thm:main}, following the approach of \cite{DF_rosini,DF_fagioli_rosini,DF_fagioli_rosini_russo,DF_fagioli_radici,radici_stra}, see also \cite{holden1,holden2}. Roughly speaking, for a given $\rho_0\in \Prob(\R)$, 
we obtain the unique dissipative measure solution $\rho$ with initial condition $\rho_0$ provided in Theorem \ref{thm:main} as the limit (in a suitable measure topology) of a measure $\rho^N$ constructed out of a set of moving particles $x_0(t),\ldots,x_N(t)$ obeying a suitable system of ODEs.

Consider an initial condition $\rho_0$ be in $\Prob(\R)$. For simplicity, we shall assume that $\rho_0$ is compactly supported. Considering the case of non compactly supported initial measures brings unnecessary technical difficulties to the problem, for which we refer to \cite{DF_fagioli_rosini}, in which the case $\rho_0\in \Prob_p(\R)$ is considered for local conservation laws. For a fixed $N\in \mathbb{N}$, we define the initial particles distribution as follows. We set $m_N=1/N$ and define $\overline{X}=T(\rho_0)=X_{\rho_0}$ according to \eqref{eq:T_map}. We now set
\begin{equation}\label{eq:particles_initial}
    \overline{x}_i=\overline{X}(i\, m_N)\,,\qquad i=0,\ldots,N\,.
\end{equation}
Note that in case $\rho_0$ has no atomic part, then $\overline{X}$ is strictly increasing on $[0,1]$ and
\begin{align*}
    & \int_{\overline{x}_i}^{\overline{x}_{i+1}}\rho_0 (y)=\int_{\R}\mathbf{1}_{[\overline{x}_i,\overline{x}_{i+1}]}(y) d\rho_0(y)=\int_0^1 \mathbf{1}_{[\overline{x}_i,\overline{x}_{i+1}]}(\overline{X}(z))dz \\
    & \ = \mathrm{meas}\left(\left\{z\in [0,1]\,:\,\, \overline{x}_i\leq \overline{X}(z)\leq \overline{x}_{i+1}\right\}\right) = (i+1)m_N-i m_N=m_N\,. 
\end{align*}
The latter is the way this scheme is usually introduced. Our definition \eqref{eq:particles_initial} is more general and allows to consider measures $\rho_0$ with atomic parts. Clearly, for a measure $\rho_0$, two consecutive particles $\overline{x}_i$ and $\overline{x}_{i+1}$ may overlap.

We now define the way particles evolve in time. The key to define the velocity $\dot{x}_i(t)$ of each particles is to find a suitable discrete version $W[\rho^N](x)$ of the velocity field $W[\rho]$ in \eqref{eq:Wh} and set
\[\dot{x}_i(t)=W[\rho^N](x_i(t)\,.\]
The main guiding principle to construct $W[\rho^N]$ (quite often in this setting) is that of replacing $\rho$ in \eqref{eq:Wh} by the empirical measure
\[\mu^N(t)=m_N\sum_{i=1}^{N}\delta_{x_i(t)}\,,\]
which would return the velocity field
\[W[\rho^N](x)=v\left(m_N \sum_{k=1}^N V(x-x_k(t))\right)\,.\]
The possible discontinuity of $V$ at zero of $W[\rho^N]$ is a possible difficulty for this problem. Even more challenging is the fact that more than one particles may occupy the same position initially. To sort both issues, somehow inspired by the procedure that led to \eqref{eq:pseudo1}, recalling that $\lambda=V(0^-)$, we define our particle system as follows
\begin{equation}\label{eq:DPA_V}
\begin{cases}
  \displaystyle{\dot{x}_i(t) = v\left( m_N \sum_{k=i+1}^N V(x_i-x_j)\right)}\\
    x_i(0) = \overline{x}_i,
\end{cases}
\end{equation}
for all $i=0,\ldots,N$. We observe that the ODEs in \eqref{eq:DPA_V} are equivalent to
\[
\dot{x}_i(t) = v\left( m_N \sum_{x_k>x_i} V(x_i-x_j)+m_N\lambda \left|\left\{k\in\{1,\ldots,N\}\,:\,\, \hbox{$x_k(t)=x_i(t)$ and $k>i$}\right\}\right|\right)\,.
\]
Roughly speaking, we impose that the movement of every particle $x_i$ is only affected by particles with position strictly larger than $x_i$ and by those with same position indexed by $k>i$. Such a mechanism is meant to label particles occupying the same initial position in such a way to provide them with an order. We will now rewrite \eqref{eq:DPA_V} in such a way to ensure a suitable existence and uniqueness theory for it. Provided that $x_i(t)\leq x_{i+1}(t)$ for all $i=0,\ldots,N-1$ and for all $t\geq 0$, recalling the definition of $U$, we compute
\begin{align}
& m_N \sum_{k=i+1}^N V(x_i-x_k)  = m_N \sum_{i=0}^N U(x_i-x_k) - m_N\sum_{i=0}^i \lambda\nonumber\\
& \ = m_N \sum_{i=0}^N U(x_i-x_k) - m_N(i+1)\lambda\,.\label{eq:equivalence}
\end{align}
Hence, \eqref{eq:DPA_V} can be rewritten as
\begin{equation}\label{eq:DPA_U}
\begin{cases}
    \displaystyle{\dot{x}_i(t) = v\left( m_N \sum_{k=0}^N U(x_i(t)-x_k(t)) -m_N \lambda\,(i+1)\right)}\\
    x_i(0) = \overline{x}_i
\end{cases}\qquad i=0,\ldots,N\,.
\end{equation}
Since $v$ and $U$ are globally Lipschitz continuous, \eqref{eq:DPA_U} has a unique global-in-time solution. We will prove that such a solution solves \eqref{eq:DPA_V} as well. Our particle approximation result is stated in the next Theorem.

\begin{thm}\label{thm:DPA}
    Let $\rho_0\in \Prob(\R)$ have compact support and let $T\geq 0$. For fixed $N\in \N$ define $\overline{x}_1,\ldots,\overline{x}_N$ as in \eqref{eq:particles_initial} and let $x_1(t),\ldots,x_N(t)$ be the unique global solution to \eqref{eq:DPA_U}. Then, $x_i(t)<x_{i+1}(t)$ for all $t>0$ and the particles $x_i(t)$, $i=1,\ldots,N$, solve \eqref{eq:DPA_V} too. Moreover, given
    \begin{equation}\label{eq:reconstruction}
        \rho^N(x,t)=\sum_{i=0}^N \rho_i(t)\mathbf{1}_{I_i(t)}\,,\quad \rho_i(t)=\frac{m_N}{|I_i(t)|}\,,\quad I_i(t)=x_{i+1}(t)-x_i(t)\,,
    \end{equation}
    we have that $\rho^N(t)$ is uniformly bounded in $L^\infty(\R)$ with respect to $N$ for all times $t>0$ and $\rho^N\rightarrow \rho$ in the weak star $L^\infty(\R\times [\delta,+\infty))$ topology for all $\delta>0$, with $\rho$ being the unique dissipative measure solution to \eqref{eq:main} in the sense of Definition \ref{def:dissipative} with initial datum $\rho_0$.
\end{thm}
We will prove Theorem \ref{thm:DPA} in Section \ref{sec:DPA}.

\begin{rem}
    \emph{
    The above Theorem \ref{thm:DPA} may apparently seem as an obvious consequence of the stability result stated in Theorem \ref{thm:main}, because clearly the particle approximation $\rho^N(\cdot,0)$ converges to the initial condition $\rho_0$ as $N\rightarrow+\infty$ in the $p$-Wasserstein distance for all $p\in [1,+\infty]$. However, the discrete measure $\rho^N(\cdot,t)$ is \emph{not} an exact solution to \eqref{eq:main} for positive times. Neither is its empirical measure version \eqref{eq:empirical}. The latter would solve \eqref{eq:main} in case $V$ is smooth on the whole $\R$. This is not the case here, due to the discontinuity of $V$ which forces initial particles to \emph{not remain particles} due to the measure-to-$L^\infty$ smoothing effect.
    }
\end{rem}

\section{Existence and uniqueness of dissipative measure solutions}\label{sec:proof_main}

This section is mostly devoted to providing the proof of Theorem \ref{thm:main} and prove some additional results. Let $p\in [1,+\infty]$. We define the operator $\mathcal{A}:L^p([0,1])\rightarrow L^p([0,1])$ as follows. Let $X\in L^p([0,1])$. Then, 
\begin{align*}
    & \mathcal{A}[X](z)=v(W_X(z))
    & \hbox{with}\qquad W_X(z)=\int_0^1 U(X(z)-X(\zeta))d\zeta -\lambda z\,.
\end{align*}
We observe that the operator $\mathcal{A}$ is well-posed from $L^p([0,1])$ onto itself in view of assumptions (v2) and (V3).

\begin{prop}\label{prop:operator}
The operator $\mathcal{A}$ is Lipschitz continuous on $L^p([0,1])$ with Lipschitz constant
\begin{equation}\label{eq:lipschitz_constant}
 C=\begin{cases}
     2^{\frac{p+1}{p}}[v]_{\mathrm{Lip}}[V]_{\mathrm{Lip}((-\infty,0])} & \hbox{if $p<+\infty$}\\
     2 [v]_{\mathrm{Lip}}[V]_{\mathrm{Lip}((-\infty,0])} & \hbox{if $p=+\infty$}\,.
 \end{cases}   
\end{equation}
\end{prop}

\begin{proof}
Assume first $p$ finite. Let $X_1,X_2\in L^p([0,1])$. We have
\begin{align*}
    & \left|v(W_{X_1}(z))-v(W_{X_2}(z))\right|\leq [v]_{\mathrm{Lip}}|W_{X_1}(z)-W_{X_2}(z)|\\
    & \ \leq [v]_{\mathrm{Lip}} \int_0^1\left| U(X_1(z)-X_1(\zeta))-U(X_2(z)-X_2(\zeta))\right|d\zeta\\
    & \ \leq [v]_{\mathrm{Lip}}[U]_{\mathrm{Lip}} \left(\left|X_1(z)-X_2(z)\right| + \int_0^1\left|X_1(\zeta)-X_2(\zeta)\right|d\zeta\right).
\end{align*}
Hence,
\begin{align*}
    & \|\mathcal{A}[X_1]-\mathcal{A}[X_2]\|_{L^p([0,1])}^p = \int_0^1 \left|v(W_{X_1}(z))-v(W_{X_2}(z)\right|^p dz\\
    & \ \leq \left([v]_{\mathrm{Lip}}[U]_{\mathrm{Lip}}\right)^p\int_0^1 \left(\left|X_1(z)-X_2(z)\right| + \int_0^1\left|X_1(\zeta)-X_2(\zeta)\right|d\zeta\right)^p dz\\
    & \ \leq 2^p\left([v]_{\mathrm{Lip}}[U]_{\mathrm{Lip}}\right)^p \left(\int_0^1\left|X_1(z)-X_2(z)\right|^p dz +\int_0^1 \left(\int_0^1 \left|X_1(\zeta)-X_2(\zeta)\right|d\zeta\right)^p dz\right)\\
    & \ \leq 2^{p+1} \left([v]_{\mathrm{Lip}}[U]_{\mathrm{Lip}}\right)^p \int_0^1\left|X_1(z)-X_2(z)\right|^p dz
\end{align*}
where the last step is justified by Jensen's inequality. The case $p=+\infty$ easily follows from sending $p\rightarrow+\infty$ and is left as an exercise.
\end{proof}

\begin{prop}\label{prop:existence1}
    Let $\overline{X} \in L^p([0,1])$. Then, there  exists one and only one solution 
    \[X:[0,1]\times [0,+\infty)\rightarrow \R\,,\qquad X\in \mathrm{Lip}_{\mathrm{loc}}([0,+\infty);L^p([0,1]))\,,\] 
    to the initial value problem \eqref{eq:pseudo1}-\eqref{eq:pseudo1_initial}.
\end{prop}

\begin{proof}
    The proof follows immediately from Theorem \ref{thm:CLP}.
\end{proof}

We now want to prove that the unique solution $X$ provided in the previous Theorem is \emph{strictly increasing} in $z$ for all times $t>0$ provided the initial condition is non-decreasing.

\begin{prop}\label{prop:X_increasing}
    Assume $\overline{X}$ is non decreasing on $[0,1]$. Then, for all $t>0$, the solution $X$ provided in Proposition \ref{prop:existence1} is strictly increasing in $z$ and the following estimate holds
\begin{equation}\label{eq:slope_X}
    X(z_2,t)-X(z_1,t)\geq \frac{b\lambda (z_2-z_1)}{[v]_{\mathrm{Lip}}[U]_{\mathrm{Lip}}}\left(1-e^{-[v]_{\mathrm{Lip}}[U]_{\mathrm{Lip}}t}\right)\,,\quad 0\leq z_1<z_2\leq 1\,.
\end{equation}
\end{prop}

\begin{proof}
    For all $t\geq 0$ and for almost every $(z_1,z_2)\in [0,1]^2$ with $z_1<z_2$, we have
\begin{align*}
    & \partial_t (X(z_2,t)-X(z_1,t)) \\
    & \ =  v\left(\int_0^1 U(X(z_2,t)-X(\zeta,t))d\zeta -\lambda z_2\right)-v\left(\int_0^1 U(X(z_1,t)-X(\zeta,t))d\zeta -\lambda z_1\right)\\
    & \ = v\left(\int_0^1 U(X(z_2,t)-X(\zeta,t))d\zeta -\lambda z_2\right)-v\left(\int_0^1 U(X(z_1,t)-X(\zeta,t))d\zeta -\lambda z_2\right)\\
    & \ +  v\left(\int_0^1 U(X(z_1,t)-X(\zeta,t))d\zeta -\lambda z_2\right)-v\left(\int_0^1 U(X(z_1,t)-X(\zeta,t))d\zeta -\lambda z_1\right)\\
    & \ \geq -[v]_{\mathrm{Lip}}[U]_{\mathrm{Lip}}\left|X(z_2,t)-X(z_1,t)\right| +b\lambda (z_2-z_1)
\end{align*}
where we have used (v1) and (v2). We claim that $X(z_2,t)\geq X(z_1,t)$ for all $t\geq 0$. Without restriction due the $C^1$ regularity of $X$ in time, assume there exist $0\leq t_1<t_2$ such that $X(z_2,t_1)=X(z_1,t_1)$ and $X(z_2,t)<X(z_1,t)$ for all $t\in (t_1,t_2)$. For $h>0$ small enough, we have, after time integration on $t\in [t_1,t_1+h]$,
\begin{align*}
    & \frac{X(z_2,t_1+h)-X(z_1,t_1+h)}{h}\geq -[v]_{\mathrm{Lip}}[U]_{\mathrm{Lip}}\frac{1}{h}\int_{t_1}^{t_1+h}\left(X(z_1,t)-X(z_2,t)\right)dt + b\lambda (z_2-z_1)\,.
\end{align*}
By sending $h\searrow 0$, recalling that $X(z_1,t)-X(z_2,t)=0$ at $t=t_1$ and in view of the time regularity of $X$, we obtain
\[\partial_t (X(z_2,t_1)-X(z_1,t_1)) \geq b\lambda (z_2-z_1) >0\,,\]
which contradicts the fact that $X(z_2,t)-X(z_1,t)$ is $C^1$ and non-increasing at time $t_1$. This proves $X(z_2,t)\geq X(z_1,t)$ for all $t\geq 0$ and for almost every $(z_1,z_2)\in [0,1]^2$ with $z_1<z_2$. This implies that for all $t\geq 0$ there exists a right-continuous representative of $X(\cdot,t)$ which is non decreasing. From now on, we shall identify $X(\cdot,t)$ with said representative.
To prove strict monotonicity, we can re-write the above estimate as
\begin{align*}
    & \partial_t (X(z_2,t)-X(z_1,t))\geq -[v]_{\mathrm{Lip}}[U]_{\mathrm{Lip}}\left(X(z_2,t)-X(z_1,t)\right) +b\lambda (z_2-z_1)\,.
\end{align*}
A simple comparison principle then yields
\begin{equation}\label{eq:comparison_smoothing}
    X(z_2,t)-X(z_1,t) \geq e^{-[v]_{\mathrm{Lip}}[U]_{\mathrm{Lip}}t}\left(X(z_2,0)-X(z_1,0)\right)+\frac{b\lambda (z_2-z_1)}{[v]_{\mathrm{Lip}}[U]_{\mathrm{Lip}}}\left(1-e^{-[v]_{\mathrm{Lip}}[U]_{\mathrm{Lip}}t}\right)\,,
\end{equation}
which easily implies \eqref{eq:slope_X}.
\end{proof}

As a consequence of Proposition \ref{prop:X_increasing}, $X(\cdot,t)$ has an  additional $BV_{\mathrm{loc}}$ regularity. Proposition \ref{prop:X_increasing} also implies a uniform $L^\infty$ estimate for positive times for the measure $\rho(\cdot,t)=T^{-1}(X(\cdot,t))$, where $T$ is the isometry defined in \eqref{eq:T_map}.

\begin{prop}\label{prop:Linfty}
Let $\rho(\cdot,t)=T^{-1}(X(\cdot,t))$ where $T$ is the isometry defined in \eqref{eq:T_map}. Then,
\begin{equation}\label{eq:Linfty_rho}
   \|\rho(\cdot,t)\|_{L^\infty(\R)}\leq \frac{[v]_{\mathrm{Lip}}[U]_{\mathrm{Lip}}}{b\lambda}\left(1-e^{-[v]_{\mathrm{Lip}}[U]_{\mathrm{Lip}}t}\right)^{-1}\,,\qquad\hbox{for all $t>0$}\,. 
\end{equation}
\end{prop}

\begin{proof}
For $t>0$, set $\rho(\cdot,t)=T^{-1}(X(\cdot,t))$ and let 
\[F(x,t)=\int_{-\infty}^x d\rho(\cdot,t) dy\,.\]
Since $X(\cdot,t)$ is strictly increasing, then $F(\cdot,t)$ is continuous on $\R$. Assume $x_1,x_2\in \R$ with $x_1<x_2$ and $F(\cdot,t)$ is strictly increasing on $(x_1,x_2)$. Then the following estimate easily follows
\begin{equation}\label{eq:slope_F}
    \frac{F(x_2,t)-F(x_1,t)}{x_2-x_1}\leq \frac{[v]_{\mathrm{Lip}}[U]_{\mathrm{Lip}}}{b\lambda}\left(1-e^{-[v]_{\mathrm{Lip}}[U]_{\mathrm{Lip}}t}\right)^{-1}
\end{equation}
as a consequence of \eqref{eq:slope_X}, using that $X(\cdot,t)$ restricted on $(F(x_1,t),F(x_2,t))$ is the inverse of $F(\cdot,t)$ on $(x_1,x_2)$. If $F(\cdot,t)$ is not strictly increasing on $(x_1,x_2)$, then \eqref{eq:slope_F} is clearly still satisfied. Then, \eqref{eq:slope_F} implies that the weak $L^\infty$ derivative $\rho(\cdot,t)=F_x(\cdot,t)$ satisfies the assertion.
\end{proof}

As a consequence of Proposition \ref{prop:Linfty} and due to the assumptions on $v$, the velocity field $v(W[\rho(\cdot,t)])$ is locally bounded and Lipschitz continuous for positive times and we can perform a very convenient change of variable to recover the weak measure formulation \eqref{eq:main} for $\rho$. 

\begin{prop}\label{prop:change_velocity}
    Let $\rho(\cdot,t)=T^{-1}(X(\cdot,t))$ where $T$ is the isometry defined in \eqref{eq:T_map}. Then, for $t>0$ the velocity field
    \begin{equation}\label{eq:velocity_field}
        G(x,t)=v(W[\rho(\cdot,t)])(x) = v\left(\int_\R U(x-y)d\rho(y,t) -\lambda F(x,t)\right)
    \end{equation}
    is Lipschitz continuous with respect to $x$. Moreover, 
    \begin{equation}\label{eq:change_velocity}
        \mathcal{A}[X(\cdot,t)](z)=v\left(W[\rho(\cdot,t)]\right)(X(z,t))\,.
    \end{equation}
\end{prop}

\begin{proof}
    Since $\rho(\cdot,t)\in L^1(\R)\cap L^\infty(\R)$ for all $t>0$ from Proposition \ref{prop:Linfty}, then $U\ast \rho(\cdot,t)$ is Lipschitz continuous (recall that $U$ is Lipschitz and $\rho(\cdot,t)\in L^1(\R)$). Moreover, $F(\cdot,t)$ has its derivative $\rho(\cdot,t)$ in $L^\infty$, and therefore $F(\cdot,t)\in \mathrm{Lip}(\R)$. Since $v$ is Lipschitz by assumption, then $G(\cdot,t)$ is the composition of two Lipschitz functions and hence is Lipschitz. To prove the second statement, we compute
    \begin{align*}
         v\left(W[\rho(\cdot,t)]\right)(X(z,t)) =   v\left(\int_\R U(X(z,t)-y)d\rho(y,t) - \lambda F(X(z,t),t) \right)\,.
    \end{align*}
    Since $X(\cdot,t)$ is strictly increasing, we have the identity $F(X(z,t),t)=z$. Moreover, by changing variable $y=X(\zeta,t)$, we get
     \begin{align*}
         v\left(W[\rho(\cdot,t)]\right)(X(z,t)) =  v\left(\int_0^1 U(X(z,t)-X(\zeta,t))d\zeta -\lambda z\right) = \mathcal{A}[X(\cdot,t)](z)\,,
    \end{align*}
    which concludes the proof.
\end{proof}

We have now all we need to conclude the proof of Theorem \ref{thm:main}.
\begin{proof}[Proof of Theorem \ref{thm:main}]
Proposition \ref{prop:existence1} and \eqref{eq:change_velocity} show the existence and uniqueness of solutions to the Cauchy problem 
    \[
\begin{cases}
    \partial_t X(z,t)= v(W[\rho(\cdot,t)])(X(z,t)) & \\
    X(z,0)=\overline{X}(z)\,.
\end{cases}
    \]
    Propositions \ref{prop:X_increasing} and \ref{prop:Linfty} imply that $\rho(\cdot,t)=T^{-1}X(\cdot,t)$ belongs to $L^\infty$ for positive times. Moreover, Proposition \ref{prop:change_velocity} shows that the vector field \eqref{eq:velocity_field} is Lipschitz continuous. Therefore, Theorem \ref{thm:continuity} implies $\rho$ is a dissipative measure solution to \eqref{eq:main} in the sense of Definition \ref{def:dissipative} with $\rho_0$ as initial condition. The same Theorem implies that any other dissipative measure solution $\widetilde{\rho}$ with the same initial condition would be such that $\widetilde{X}(\cdot,t)=T\widetilde{\rho}(\cdot,t)$ satisfy \eqref{eq:pseudo1}-\eqref{eq:pseudo1_initial} with $\overline{X}=T^{-1}(\rho_0)$, but then the uniqueness of $X$ stated in Proposition \ref{prop:existence1} and the fact that $T$ is an isometry implies that $\widetilde{\rho}=\rho$. We have therefore proven the existence and uniqueness statement as well as (i). (ii) follows from Proposition \ref{prop:Linfty}. To prove \eqref{eq:stability}, we recall that as an easy consequence of Proposition \ref{prop:operator} the two quantile representations $X_{\rho_1}$ and $X_{\rho_2}$ satisfy
    \[\|X_{\rho_1}(\cdot,t)-X_{\rho_2}(\cdot,t)\|_{L^{p}([0,1])}\leq e^{Ct}\|X_{\rho_1}(\cdot,0)-X_{\rho_2}(\cdot,0)\|_{L^{p}([0,1])}\] with $C$ provided in \eqref{eq:lipschitz_constant}. In view of \eqref{eq:isometry}, we obtain the desired assertion.
\end{proof}

In case the initial condition $\rho_0\in L^\infty(\R)$ and in the special case in which $V$ is nonnegative and non decreasing on the half line $(-\infty,0]$, we can easily prove that the solution $\rho(\cdot,t)$ is uniformly bounded by the $L^\infty$ norm of the initial datum. This property was already proven in previous papers, see for example \cite{keimer_pflug_2017}. However, here we provide an alternative proof based on our quantile formulation approach. 

\begin{prop}\label{prop:Linfinity_bound}
    Assume that $V$ satisfies the additional assumption of being non-decreasing on $(-\infty,0]$ and $V>0$ on $(-\infty,0]$. Assume $\rho_0\in L^\infty(\R)$. Then, for all $t<0$,
    \[\|\rho(\cdot,t)\|_{L^\infty(\R)}\leq \|\rho_0\|_{L^\infty(\R)}\,.\]
\end{prop}

\begin{proof}
Based on the result in Theorem \ref{thm:main}, the quantile function $X(\cdot,t)=T_{\rho(\cdot,t)}$ satisfies \eqref{eq:pseudo_main}. Since the solution $\rho(\cdot,t)$ is in $L^\infty$ for all positive times due to \eqref{eq:smoothing}, the function $X(\cdot,t)$ is strictly increasing on $[0,1]$. Hence, recalling the definition of $U$, \eqref{eq:pseudo_main} becomes
\begin{align}
    & \partial_t X(z,t)=v\left(\int_0^z U(X(z,t)-X(\zeta,t))d\zeta + \int_z^1 U(X(z,t)-X(\zeta,t))d\zeta -\lambda z\right)\nonumber\\
    & \ =v\left(\int_0^z \lambda d\zeta + \int_z^1 U(X(z,t)-X(\zeta,t))d\zeta -\lambda z\right)= v\left(\int_z^1 V(X(z,t)-X(\zeta,t))d\zeta\right)\,.\label{eq:reduced}
\end{align}
For a fixed $h>0$ and for every $z\in [0,1-h]$, consider the ratio
\[Y_h(z,t)=\frac{X(z+h,t)-X(z,t)}{h}\,.\]
For $t=0$ and for almost every $z\in [0,1-h]$, since $\overline{X}$ is strictly increasing (due to the fact that $\overline{X}=X_{\rho_0}$ and $\rho_0\in L^\infty(\R)$), we have
\[Y_h(z,0)=\frac{\overline{X}(z+h)-\overline{X}(z)}{h}\geq R^{-1}\]
with $R=\|\rho_0\|_{L^\infty(\R)}$. This is easily seen from the estimate
\[\frac{F_{\rho_0}(x+k)-F_{\rho_0}(x)}{k}\leq R\,,\quad x\in \R\,,\quad k>0\,.\]
Our goal is to prove that $Y_h(z,t)\geq R^{-1}$ for all $z\in [0,1-h]$, for all $h>0$, and for all $t\geq 0$. Let $t_1\geq 0$ be the first time at which $Y_h(z_0,t_1)=R^{-1}$ for some $z_0\in [0,1-h]$ and $Y_h(z_0,t)<R^{-1}$ for $t\in (t_1,t_2)$ for some $t_2>t_1$. Clearly, $Y_h(z,t_1)\geq Y_h(z_0,t_1)$ for all $z\in [0,1-h]$. Such inequality implies
\[
\frac{X(z_0+h,t_1)-X(z+h,t_1)}{h}\leq \frac{X(z_0,t_1)-X(z,t_1)}{h}\,.
\]
Since $V$ is non-decreasing, we get
\[V\left(X(z_0+h,t_1)-X(z+h,t_1)\right)\leq V\left(X(z_0,t_1)-X(z,t_1)\right)\,.\]
Integrating on $z\in [z_0,1-h]$ we obtain
\begin{align*}
    & \int_{z_0+h}^{1} V(X(z_0+h,t_1)-X(\zeta,t))d\zeta = \int_{z_0}^{1-h} V(X(z_0+h,t_1)-X(z+h,t))dz\\
    & \ \leq  \int_{z_0}^{1-h}V(X(z_0,t_1)-X(z,t_1))dz< \int_{z_0}^{1}V(X(z_0,t_1)-X(z,t_1))dz
\end{align*}
where we have used $V> 0$. Since $v$ is decreasing, we get
\[v\left( \int_{z_0+h}^{1} V(X(z_0+h,t_1)-X(\zeta,t_1))d\zeta\right)> v\left(\int_{z_0}^{1}V(X(z_0,t_1)-X(\zeta,t_1))d\zeta\right)\]
which implies
\[\partial_t Y_h(z_0,t)|_{t=t_1}=\partial_t \frac{X(z_0+h,t)-X(z_0,t)}{h}|_{t=t_1}> 0\]
which contradicts the assumptions on $t_1$. By the arbitrariness of $h>0$, all ratios
\[\frac{X(z_2,t)-X(z_1,t)}{z_2-z_1}\]
are controlled from below by $R^{-1}$ for all $0\leq z_1<z_1\leq 1$ and for all $t\geq 0$, which implies
\[\frac{F_{\rho(\cdot,t)}(x_2)-F_{\rho(\cdot,t)}(x_1)}{x_2-x_1}\leq R\]
for all $x_1,x_2\in \R$ with $x_1<x_2$ and for all $t\geq 0$, which in turn implies
\[\|\rho(\cdot,t)\|_{L^\infty(\R)}\leq R\]
as desired.
\end{proof}


In case $\rho_0\in L^\infty$ but $V$ is \emph{not} necessarily monotone and positive, we can still prove a uniform-in-time $L^\infty$ bound under an additional threshold assumption.

\begin{prop}\label{prop:Linfinity_bound2}
    Assume the initial condition $\rho_0$ satisfies 
    \begin{equation}\label{eq:threshold}
        \|\rho_0\|_{L^\infty(\R)}\leq \frac{[v]_{\mathrm{Lip}}[V]_{\mathrm{Lip}((-\infty,0])}}{b\lambda}\,.
    \end{equation}
    Then, for all $t>0$,
    \[\|\rho(\cdot,t)\|_{L^\infty(\R)}\leq \frac{[v]_{\mathrm{Lip}}[V]_{\mathrm{Lip}((-\infty,0])}}{b\lambda}\,.\]
\end{prop}

\begin{proof}
    From the proof of Proposition \ref{prop:X_increasing}, given $0\leq z_1<x_2\leq 1$, we deduce inequality \eqref{eq:comparison_smoothing}, which can be written as
    \[
    \frac{X(z_2,t)-X(z_1,t)}{z_2-z_1} \geq e^{-[v]_{\mathrm{Lip}}[U]_{\mathrm{Lip}}t}\left(\frac{X(z_2,0)-X(z_1,0)}{z_2-z_1}-\frac{b\lambda }{[v]_{\mathrm{Lip}}[U]_{\mathrm{Lip}}}\right)+\frac{b\lambda }{[v]_{\mathrm{Lip}}[U]_{\mathrm{Lip}}}\,.
    \]
    The assumption \eqref{eq:threshold} implies
    \[
    \left(\frac{X(z_2,0)-X(z_1,0)}{z_2-z_1}-\frac{b\lambda }{[v]_{\mathrm{Lip}}[U]_{\mathrm{Lip}}}\right)\geq 0\,,
    \]
    which provides
    \[
     \frac{X(z_2,t)-X(z_1,t)}{z_2-z_1}\geq \frac{b\lambda }{[v]_{\mathrm{Lip}}[U]_{\mathrm{Lip}}}
    \]
    and the assertion.
\end{proof}

\section{Deterministic Particle Approximation}\label{sec:DPA}

In this Section we provide some estimates for the particle scheme \eqref{eq:DPA_V} and we prove Theorem \ref{thm:DPA}. We shall work first on the solution $x_1(t),\ldots,x_N(t)$ to \eqref{eq:DPA_U}. By mimicking the results obtained in the Section \ref{sec:proof_main}, with specific reference to \eqref{eq:slope_X}, we obtain the following lemma, which is somehow a particle version of \eqref{eq:slope_X}.

\begin{lem}\label{l:reg_eff}
Let $x_0(t),\ldots,x_N(t)$ be the unique global in time solution to \eqref{eq:DPA_U}. Then, for all $0=1,\ldots,N-1$ we have
  \begin{equation}\label{eq:disc_max}
      x_{i+1}(t)-x_{i}(t)\geq \frac{b\lambda m_N} {2 [v]_{\mathrm{Lip}}[V]_{\mathrm{Lip}((-\infty,0])}}\left(1-e^{-2[v]_{\mathrm{Lip}}[V]_{\mathrm{Lip}((-\infty,0])}t}\right),\quad\mbox{ for all }t> 0.
  \end{equation}  
\end{lem}
\begin{proof}
For all $i=0, \ldots,N-1$, we use assumptions (v1) and (v2) and compute
\begin{align*}
    &  \dot{x}_{i+1}(t)-\dot{x}_{i}(t) = v\left( m_N \sum_{j=0}^N U(x_{i+1}-x_j)-m_N\lambda (i+2)\right)\\
    & \ -v\left( m_N \sum_{j=0}^N U(x_i-x_j)-m_N\lambda (i+1)\right)\\
      &\ = v\left( m_N \sum_{j=0}^N U(x_{i+1}-x_j)-m_N\lambda (i+2)\right)-v\left( m_N \sum_{j=0}^N U(x_{i+1}-x_j)-m_N\lambda (i+1)\right)\\
     & + v\left( m_N \sum_{j=0}^N U(x_{i+1}-x_j)-m_N\lambda (i+1)\right)-v\left( m_N \sum_{j=0}^N U(x_i-x_j)-m_N\lambda (i+1)\right)\\
     & \ \geq b\lambda m_N+ v\left( m_N \sum_{j=0}^N U(x_{i+1}-x_j)-m_N\lambda (i+1)\right)-v\left( m_N \sum_{j=0}^N U(x_i-x_j)-m_N\lambda (i+1)\right)\\
     & \ \geq b\lambda m_N - \frac{N+1}{N} [v]_{\mathrm{Lip}}[V]_{\mathrm{Lip}((-\infty,0])}|x_{i+1}(t)-x_i(t)|.
\end{align*}
Now, we know that $x_{i+1}(0)-x_i(0)\geq 0$ for all $i=0,\ldots,N-1$. We claim that $x_{i+1}(t)-x_i(t)\geq 0$ for all $i=0,\ldots,N-1$ and for all $t \geq 0$. Assume this is not true, that is, there exist $0\leq t_1<t_2$ such that $x_{i+1}(t_1)-x_i(t_1)=0$ and $x_{i+1}(t)-x_i(t)<0$ on $t\in (t_1,t_2)$. By integrating both sides of the above estimate on $[t_1,t_1+h)$ with $0<h<(t_2-t_1)$ we obtain
\begin{align*}
    & \frac{x_{i+1}(t_1+h)-x_i(t_1+h)}{h}\geq b\lambda m_N-\frac{N+1}{N} [v]_{\mathrm{Lip}}[V]_{\mathrm{Lip}((-\infty,0])} \frac{1}{h}\int_{t_1}^{t_1+h} |x_{i+1}(s)-x_i(s)| ds
\end{align*}
and since $x_{i+1}(t_1)-x_i(t_1)=0$ and all $x_i$, $i=0,\ldots,N$ are continuous functions, we get, as $h\searrow 0$,
\[\frac{d^+}{dt}(x_{i+1}(t)-x_i(t))|_{t=t_1} \geq b\lambda m_N\]
which contradicts the fact that $x_{i+1}(t)-x_i(t)$ does not increase at $t=t_1$. Consequently, we get, for $N\geq 1$,
\begin{align}
    &  \dot{x}_{i+1}(t)-\dot{x}_{i}(t) \geq b\lambda m_N - 2 [v]_{\mathrm{Lip}}[V]_{\mathrm{Lip}((-\infty,0])}(x_{i+1}(t)-x_i(t))\,.\label{eq:comparison_particles}
\end{align}
Hence, a simple comparison argument implies the assertion.
\end{proof}

The result in Lemma \ref{l:reg_eff} implies that particles maintain the same order for all times. Moreover, even if initially some of them overlap, they instantaneously detach from each other, that is $x_i(t)<x_{i+1}(t)$ for all $t>0$. Recalling the computation \eqref{eq:equivalence}, we obtain as a consequence that $x_0(t),\ldots,x_N(t)$ satisfy the ODE system \eqref{eq:DPA_V}. Moreover, we are now in a position to define the discrete densities $\rho_i(t)$ and the density reconstruction $\rho^N$ as in \eqref{eq:reconstruction}, namely
\begin{equation}\label{eq:disc_den2}
\rho^N(x,t)=\sum_{i=0}^{N-1}\rho_i(t)\mathbb{1}_{I_i(t)}(x),\quad \rho_i(t) = \frac{m_N}{|I_i|(t)}, \quad\mbox{ for all }t> 0,
\end{equation}
with
\begin{equation}\label{eq:interv2}
    I_i(t)=\left[x_i(t),x_{i+1}(t)\right), \quad|I_i|(t)=x_{i+1}(t)-x_{i}(t),\quad\mbox{ for all } i=0,\ldots,N-1, \mbox{ and }t> 0\,.
\end{equation}
As a consequence of Lemma \ref{l:reg_eff}, we get
\begin{equation}\label{eq:linfty_b}
\rho_i(t)\leq \frac{2[v]_{\mathrm{Lip}}[V]_{\mathrm{Lip}((-\infty,0])}}{b\lambda\left(1-e^{-2[v]_{\mathrm{Lip}}[V]_{\mathrm{Lip}((-\infty,0])}t}\right)}, \quad\mbox{ for all }t > 0.
\end{equation}


The next result provides an improved version of Lemma \ref{l:reg_eff} under the same assumptions of Proposition \ref{prop:Linfinity_bound}.

\begin{prop}\label{prop:maximum}
    Let $\rho_0\in L^\infty(\R)$ and denote $R=\|\rho_0\|_{L^\infty(\R)}$. Assume $V$ satisfies the additional assumptions
    \begin{itemize}
        \item [(i)] $V(x)>0$ for all $x\leq 0$,
        \item [(ii)] $V$ is non-decreasing on $(-\infty,0]$.
    \end{itemize}
    Then, for all $t>0$,
    \begin{equation*}
        x_{i+1}(t)-x_{i}(t)\geq \frac{m_N}{R}\quad \mbox{ for all }t\geq 0,
    \end{equation*}    
    and consequently, for all $t\geq 0$,
    \[\|\rho^N(\cdot,t)\|_{L^\infty(\R)}\leq R\,.\]
\end{prop}
\begin{proof}
Note that the condition is automatically satisfied for $t=0$ by construction of the initial particles distribution where $\bar{x}_i<\bar{x}_{i+1}$ for all $i=0,\ldots,N-1$. Assume that $t_1>0$ is the first time for which there exists at least one index $i\in\left\{0,\ldots,N\right\}$ such that
    \begin{equation*}
        x_{i+1}(t_1)-x_{i}(t_1)= \frac{m_N}{R}.
    \end{equation*}
In $t_1$ we have that $x_{i+1}(t_1) - x_i(t_1) \leq x_{j+1}(t_1) - x_j(t_1)$ for every $j \neq i$. In particular, for $j>i$, we have that 
\begin{align*}
    x_{i+1}(t_1) - x_{j+1}(t_1) =& \sum_{k = i+1}^{j-1} \big(x_{k}(t_1) - x_{k+1}(t_1))\big) \, + \, \big(x_{j}(t_1) - x_{j+1}(t_1))\big) \\
    \leq& \sum_{k = i+1}^{j-1} \big(x_{k}(t_1) - x_{k+1}(t_1))\big) \, + \, \big(x_i(t_1) - x_{i+1}(t_1)\big) = x_i(t_1) - x_{j}(t_1) 
\end{align*}
and since $V$ is monotone increasing on $(-\infty,0]$ we deduce that 
\[  V( x_{i+1}(t_1) - x_{j+1}(t_1)) \leq V (x_i(t_1) - x_{j}(t_1)) \quad \text{for all} \; j > i. \]
As a consequence,
\begin{align*}
    \sum_{j = i+1}^N V( x_{i+1}(t_1) - x_{j}(t_1)) &= \lambda + \sum_{j=i+2}^N V ( x_{i+1}(t_1) - x_{j}(t_1)) \\
    &= \lambda + \sum_{j = i+1}^{N-1} V ( x_{i+1}(t_1) - x_{j+1}(t_1)) \\
    &\leq \lambda + \sum_{j = i+1}^{N-1} V (x_i(t_1) - x_j(t_1)) \\
    &= \sum_{j=i}^{N-1} V (x_i(t_1) - x_j(t_1))\\
    &<\sum_{j=i}^{N} V (x_i(t_1) - x_j(t_1))
\end{align*}
where we have used $V>0$. Hence, the strict monotonicity of $v$ implies that
\begin{equation*}
    v \left( m_N\sum_{j= i+1}^N V( x_{i+1}(t_1) - x_{j}(t_1))\right)> v \left( m_N\sum_{j= i}^N V (x_i(t_1) - x_j(t_1))\right).
\end{equation*}
Assume furthermore that there exists $t_2 > t_1$ such that
     \begin{equation*}
        x_{i+1}(t)-x_{i}(t)< \frac{m_N}{R},\quad\mbox{ for all }t\in\left(t_1,t_2\right].
    \end{equation*}
The existence of $t_2$ leads a contradiction. Indeed, distance between the selected consecutive particles can only increase after time $t_1$ since 
\begin{align*}
    \frac{d}{dt}\left(x_{i+1}(t)-x_{i}(t)\right)_{t=t_1} & = v\left(m_N\sum_{j\geq i+1}V\left(x_{i+1}(t_1)-x_j(t_1)\right)\right)-v\left(m_N\sum_{j\geq i}V(x_i\left(t_1)-x_j(t_1)\right)\right)>0.
\end{align*}
\end{proof}

\begin{rem}
    \emph{
    A discrete version of Proposition \ref{prop:Linfinity_bound2} can be obtained also at the particle level. Indeed, \eqref{eq:comparison_particles} implies
    \begin{align*}
    &  \frac{x_{i+1}(t)-x_{i}(t)}{m_N} \geq e^{-2 [v]_{\mathrm{Lip}}[V]_{\mathrm{Lip}((-\infty,0])}t}\left[\frac{x_{i+1}(0)-x_i(0)}{m_N} -\frac{b\lambda }{2[v]_{\mathrm{Lip}}[V]_{\mathrm{Lip}((-\infty,0])}}\right] + \frac{b\lambda }{2[v]_{\mathrm{Lip}}[V]_{\mathrm{Lip}((-\infty,0])}}\,,
\end{align*}
and hence, assuming $\rho_0(x)\leq \frac{2[v]_{\mathrm{Lip}}[V]_{\mathrm{Lip}((-\infty,0])}}{b\lambda }$ implies 
\[\rho^N(x,t)\leq \frac{2[v]_{\mathrm{Lip}}[V]_{\mathrm{Lip}((-\infty,0])}}{b\lambda }\]
for all times.
    }
\end{rem}




The next lemma shows that the support of $\rho^N$ is uniformly bounded in $N$ on compact time intervals.  
\begin{lem}\label{l:spt_max}
For every $T\geq 0$, we have
  \begin{equation}\label{eq:spt_max}
    \sup_{t \in [0,T]}  \sup_N \, (x_{N}(t)-x_{0}(t)) < \infty.
  \end{equation}  
\end{lem}
\begin{proof}
    We compute
    \begin{align*}
        \dot{x}_N(t) - \dot{x}_0(t) =& v(0) -  v \left( m_N \sum_{j > 0}  V(x_0(t) - x_j(t)) \right) \\
        \leq& [v]_{\mathrm{Lip}}\left| m_N \sum_{j >0} V(x_0(t) - x_j(t))  \right| \\
        \leq& [v]_{\mathrm{Lip}}  \| V\|_{L^\infty(\R)} = [v]_{\mathrm{Lip}} \lambda, 
    \end{align*}
    thus 
    \[x_N(t) - x_0(t) = x_N(0) - x_0(0) + \int_0^t  (\dot{x}_N(s) - \dot{x}_0(s)) \, ds \leq |\mathrm{supp}\,\rho_0| + [v]_{\mathrm{Lip}} \lambda T \]
    which is bounded uniformly with respect to $N$ and $t$.
\end{proof}
For future use let us also estimate 
\begin{align}\label{eq:Lipschitz_particles}
\notag
    \left|\dot{x}_{i+1}(t) - \dot{x}_i(t) \right|=& \left|v \left( m_N \sum_{j > i+1} V(x_{i+1}(t) - x_j(t)) \right) - v \left( m_N \sum_{j > i} V(x_{i}(t) - x_j(t)) \right)\right| \\
    \notag
    \leq& [v]_{\mathrm{Lip}} m_N \left| \sum_{j > i+1} V(x_{i+1}(t) - x_j(t)) - \sum_{j > i} V(x_{i}(t) - x_j(t))  \right| \\
    \notag\leq&  [v]_{\mathrm{Lip}} \Big( m_N |V(x_i(t)-x_{i+1}(t))| + [V]_{\mathrm{Lip}((-\infty,0])} (x_{i+1}(t) - x_i(t)) \Big)\\
    \leq&  [v]_{\mathrm{Lip}}  \Big( m_N (\lambda +[V]_{\mathrm{Lip}((-\infty,0])} M(T)) + [V]_{\mathrm{Lip}((-\infty,0])}  (x_{i+1}(t) - x_i(t)) \Big)
\end{align}
where 
\begin{equation}\label{eq:MT}
    M(T)=\sup_{N\in \N}\sup_{0\leq t\leq T}|x_N(t)-x_0(t)|
\end{equation} 
which is finite due to \eqref{eq:spt_max}.


We are now ready to complete the proof of Theorem \ref{thm:DPA}.

\begin{proof}[Proof of Theorem \ref{thm:DPA}]
For clarity we present the proof in three steps.

\smallskip\noindent
\textbf{Step 1 - Compactness.}  Thanks to \eqref{eq:linfty_b}, the sequence $\rho^N$ is uniformly bounded in $\L\infty\left(\R \times [a,T]\right)$ for any fixed $a>0$.
Thus, $\rho^N$ weak$-*$ converges to $\rho_a\in \L\infty\left(\R \times [a,T]\right)$. By uniqueness for every $0<a_1<a_2$ we have that $\rho_{a_1}=\rho_{a_2}$ on $\R \times [a_2,T]$. Hence, there exists $\rho\in\L\infty\left(\R \times (0,T]\right)$ such that $\rho^N$ weak$-*$ converges to $\rho$ in $\R \times I$, for any compact subset $I\subset (0,T]$.

In addition, the limit $\rho$ is in $\L\infty\left((0,T];\L1(\R)\right)$. Indeed, for every $t\in (0,T]$ the sequence $\rho^N(t,\cdot)$ is compact in $\Prob_2(\R)$ in the $2-$Wasserstein topology since they are probability measures with uniformly compact supports, see \eqref{eq:spt_max}. Let $\nu :(0,T]\to \Prob_2(\R)$ be the curve of $\nu_t\in\Prob_2(\R)$ being the limits of $\rho^N(t,\cdot)$, then by uniqueness $\nu$ coincides with $\rho$ in $\L\infty\left(\R \times (0,T]\right)$. Hence $\rho\in \L\infty\left((0,T];\L1(\R)\right)$ and $\|\rho(t,\cdot)\|_{\L1(\R)}=1$ for a.e. $t\in (0,T]$.

\smallskip
\textbf{Step 2 - Convergence.}
We now prove that for every $\varphi \in C_c^{\infty}((0,T)\times\R)$
\begin{equation}\label{eq:weak_sol_disc}
  \int_0^T \int_{\R} \rho^N(x,t) \partial_t \varphi(x,t)\,dx\,dt + \int_0^T \int_{\R} \rho^N(x,t)v( V \ast \rho^N(x,t)) \partial_x \varphi(x,t) \,dx\,dt,
\end{equation}
converges to
    \begin{equation}\label{eq:weak_sol}
\int_{0}^T\int_{\R}\rho(x,t)\partial_t\varphi(x,t)\,dx\,dt+\int_{0}^T\int_{\R}\rho(x,t)v(V\ast\rho)\partial_x\varphi(x,t)\,dx\,dt\,,
    \end{equation}
as $N\to \infty$ where $\rho$ is the limit obtained in Step 1. Recall that $\rho^N$ converges weak-$*$ in $L^{\infty}(\R \times I_\varphi)$ to $\rho$, where $I_\varphi\subset (0,T)$ denotes the smallest interval containing the support of $\varphi$ in time. Thus, for every $f \in L^1(\R \times I_\varphi)$ it is true that
\[ \iint_{\R \times I_\varphi} \rho^N(x,t) f(x,t) dt\,dx \; \to \; \iint_{\R \times I_\varphi} \rho(x,t) f(x,t) dt\,dx \quad \text{as} \quad N \to \infty.  \]
The convergence of the first integral in \eqref{eq:weak_sol_disc} easy follows from the previous consideration, due to the regularity of the test function. Concerning the second integral, we estimate
\begin{align*}
&\left|\int_{0}^T\int_{\R}\rho^N(x,t)v(V\ast\rho^N)\partial_x\varphi(x,t)\,dx\,dt-\int_{0}^T\int_{\R}\rho(x,t)v(V\ast\rho)\partial_x\varphi(x,t)\,dx\,dt\right|\\
\leq& \left| \int_{0}^T\int_{\R}\rho^N(x,t) \Big(v(V\ast\rho^N)  - v(V\ast\rho)\Big)\partial_x\varphi(x,t)\,dx\,dt \right| \\
& + \left|  \int_{0}^T\int_{\R}  \Big( \rho^N(x,t) - \rho(x,t) \Big)  v(V\ast\rho)\partial_x\varphi(x,t)\,dx\,dt \right| \\
\leq& \| \rho^N\|_{L^\infty(\mathrm{supp}(\varphi))}[v]_{\mathrm{Lip}} \int_{0}^T\int_{\R}|\partial_x \varphi(x,t)| \left| \int_{\R} V(y-x) \big( \rho^N(y,t) - \rho(y,t)\big) dy \right|\, dx\, dt \\
&+ \left|  \int_{0}^T\int_{\R}  \Big( \rho^N(x,t) - \rho(x,t) \Big)  v(V\ast\rho)\partial_x\varphi(x,t)\,dx\,dt \right|.
\end{align*}
We analyse separately the last two lines of the above chain of inequalities. Notice that the last line converges to $0$ since we can use the weak-$*$ convergence with $f = v(V\ast\rho)\partial_x\varphi$ (which is an $L^1$ function since $v$ is bounded on the support of $\varphi$). 
Moreover, 
by Lemma \ref{l:spt_max} there exists a bounded set $K\subset\R$ such that $\spt(\rho^N)\cup\spt(\rho)\subset K$ and we find for a.e. $(x,t)$ in $\R\times I_\varphi$ that
\[   \int_{\R} V(y-x) \big( \rho^N(y,t) - \rho(y,t)\big) dy=\int_{K} V(y-x) \big( \rho^N(y,t) - \rho(y,t)\big) dy \, \to \, 0 \quad \text{as} \quad N \to \infty;   \]
then 
\[ |\partial_x\varphi(x,t)| \left| \int_{K} V(y-x) \big( \rho^N(y,t) - \rho(y,t)\big) dy\right|  \, \to \, 0 \quad \text{pointwise a.e. in $\R\times I_\varphi$} \]
and it is immediate to see that 
\[  |\partial_x\varphi(x,t)| \left| \int_{\R} V(y-x) \big[ \rho^N(y,t) - \rho(y,t)\big] dy\right| \leq 2 C\left(([V]_{\mathrm{Lip}((-\infty,0])},\lambda,M(T)\right)  \| \rho\|_{L^{\infty}(\R\times I_\varphi)} |\partial_x \varphi (x,t)|   \]
where $M(T)$ is defined in \eqref{eq:MT} and the r.h.s. is a measurable function on $\R \times (0,T)$. Then Lebesgue dominated convergence Theorem ensures that 
\[ \int_0^T \int_{\R} |\partial_x\varphi(x,t)| \left| \int_{\R} V(y-x) \big( \rho^N(y,t) - \rho(y,t)\big) dy\right| \, dx \, dt \,\to \, 0 \quad \text{as} \quad N \to \infty.   \]

\smallskip
\textbf{Step 3 - Consistency.}
We are left to check that $\rho^N$ are approximate solutions of \eqref{eq:weak_sol}. To do this we have to show that \eqref{eq:weak_sol_disc} vanishes as $N \to \infty$. 
Observe that
\[
\dot{\rho}_i(t)=-m_N\rho_i(t)^2(\dot{x}_{i+1}(t)-\dot{x}_i(t)) = -\rho_i(t)\frac{\dot{x}_{i+1}(t)-\dot{x}_i(t)}{x_{i+1}(t)-x_i(t)}\,,
\]
and, by applying twice an integration by parts, we may write 
\begin{align}\label{eq:partial_t}
\notag
    \int_0^T \int_{\R} &\rho^N(x,t) \partial_t \varphi(x,t) \,dx\,dt \\
    \notag
    =& \sum_{i=0}^{N-1} \int_0^T \rho_i(t) \left( (\dot{x}_{i+1}(t) - \dot{x}_i(t)) \klintmed_{x_i(t)}^{x_{i+1}(t)} \!\!\!\!\varphi(x,t)\,dx -   \big( \varphi(t,x_{i+1}(t)) \dot{x}_{i+1}(t) - \varphi(t,x_i(t)) \dot{x}_i(t) \big) \right) \, dt \\
    \notag
    =& \sum_{i=0}^{N-1} \int_0^T \rho_i(t) (\dot{x}_{i+1}(t) - \dot{x}_i(t)) \left(  \klintmed_{x_i(t)}^{x_{i+1}(t)} \!\!\!\!\varphi(x,t)\,dx - \varphi(t,x_{i+1}(t)) \right) \, dt \\
    &- \sum_{i=0}^{N-1} \int_0^T \rho_i(t) \dot{x}_i(t) (\varphi(t,x_{i+1}(t)) - \varphi(t,x_i(t)))\, dt 
\end{align}
and the first of the two terms above actually vanishes on its own when the number of particles grows to $+\infty$. Indeed, thanks to \eqref{eq:Lipschitz_particles}, 
we estimate
\begin{align*}
    &\left|\sum_{i=0}^{N-1} \int_0^T \rho_i(t) (\dot{x}_{i+1}(t) - \dot{x}_i(t)) \left(  \klintmed_{x_i(t)}^{x_{i+1}(t)} \!\!\!\!\varphi(x,t)\,dx - \varphi(t,x_{i+1}(t)) \right) \, dt\right| \\
    \leq& \sum_{i=0}^{N-1} \int_0^T \rho_i(t)  \left|\dot{x}_{i+1}(t) - \dot{x}_i(t)\right| \int_{x_i(t)}^{x_{i+1}(t)}|\partial_x \varphi(x,t)|\,dx\,dt \\
    \leq& m_N [v]_{\mathrm{Lip}} \left((\lambda+[V]_{\mathrm{Lip}((-\infty,0])} M(T))\|\rho\|_{L^\infty(\R \times I_\varphi)}+[v]_{\mathrm{Lip}}\right)\sum_{i=0}^{N-1} \int_0^T  \int_{x_i(t)}^{x_{i+1}(t)}|\partial_x \varphi(x,t)|\,dx\,dt\\
     = & m_N [v]_{\mathrm{Lip}}\left((\lambda+[V]_{\mathrm{Lip}((-\infty,0])} M(T))\|\rho\|_{L^\infty(\R \times I_\varphi)}+[V]_{\mathrm{Lip}((-\infty,0])}\right) \|\partial_x \varphi\|_{\L1\left(\R\times(0,T)\right)}
\end{align*}
which goes to $0$ with $m_N \to 0$ as $N \to \infty$. 
To conclude we are left to show that the second term in \eqref{eq:partial_t} cancels with \begin{equation*}
    \int_0^T \int_{\R} \rho^N(x,t) v( V \ast \rho^N(x,t)) \partial_x \varphi(x,t) \, dx \, dt\,
\end{equation*}
for large $N$. By substituting \eqref{eq:DPA_V} in the second term of \eqref{eq:partial_t} we have
\begin{align*}
&  \sum_{i=0}^{N-1} \int_0^T \rho_i(t) \dot{x}_i(t) (\varphi(t,x_{i+1}(t)) - \varphi(t,x_i(t)))\, dt -\int_0^T \int_{\R} \rho^N(x,t) v( V \ast \rho^N(x,t)) \partial_x \varphi(x,t) \, dx \, dt\\
    =&\sum_{i=0}^{N-1} \int_0^T \rho_i  \left(  \dot{x}_i \int_{x_i}^{x_{i+1}} \partial_x \varphi \, dx  - \int_{x_i}^{x_{i+1}} v\left( \sum_{j=0}^{N-1} \rho_j \int_{x_j}^{x_{j+1}} V(x-y) \, dy \right) \partial_x \varphi\, dx\right)\, dt \\
    =& \sum_{i=0}^{N-1} \int_0^T \rho_i \int_{x_i}^{x_{i+1}} \partial_x \varphi \left( v\left( m_N \sum_{j>i} V(x_i - x_j) \right) -  v\left( \sum_{j\geq i} \rho_j \int_{x_j}^{x_{j+1}} V(x-y) \, dy \right) \right) \, dx\, dt. \end{align*}
The Lipschitz regularity of $v$ allows to estimate the absolute value of the above quantity by
    \begin{align*}
    & [v]_{\mathrm{Lip}} \sum_{i=0}^{N-1} \int_0^T \rho_i(t) \int_{x_i(t)}^{x_{i+1}(t)} |\partial_x\varphi(x,t)|  \sum_{j>i} \rho_j(t) \int_{x_j(t)}^{x_{j+1}(t)}  |V(x_i(t) - x_j(t)) - V(x-y)| \, dy\, dx\, dt.   \\
    & + [v]_{\mathrm{Lip}} \sum_{i=0}^{N-1} \int_0^T \rho_i(t) \int_{x_i(t)}^{x_{i+1}(t)} |\partial_x\varphi(x,t)|    \rho_i(t) \int_{x_i(t)}^{x_{i+1}(t)}  | V(x-y)| \, dy\, dx\, dt:=I_1+I_2.  \end{align*}
Using now Lipschitz regularity of $V$ and the uniform estimate for the support in \eqref{eq:spt_max}, we can deduce the following bound for $I_1$
    \begin{align*}
    I_1\leq& [v]_{\mathrm{Lip}} [V]_{\mathrm{Lip}((-\infty,0])}\sum_{i=0}^{N-1} \int_0^T \rho_i(t) \int_{x_i(t)}^{x_{i+1}(t)} |\partial_x\varphi(x,t)|  \sum_{j>i} m_N \big( |x_i(t) - x| + |x_j(t) - x_{j+1}(t)| \big)\, dx\, dt \\
    \leq&  [v]_{\mathrm{Lip}} [V]_{\mathrm{Lip}((-\infty,0])} \|\partial_x \varphi\|_{L^\infty(\R \times (0,T))}  \sum_{i=0}^{N-1} \int_0^T m_N \Big(  |x_{i}(t) - x_{i+1}(t)| + m_N |x_N(t) - x_0(t)| \Big)\, dt \\
    \leq&  2[v]_{\mathrm{Lip}} [V]_{\mathrm{Lip}((-\infty,0])} \|\partial_x \varphi\|_{L^\infty(\R \times (0,T))} T  \sup_{t \in [0,T]} \sup_N |x_N(t) - x_0(t)|   m_N,
\end{align*}
which again converges to $0$ as $m_N \to 0$. Finally, for $I_2$ we have
    \begin{align*}
    I_2\leq&  (\lambda+ [V]_{\mathrm{Lip}((-\infty,0])}M(T)) m_N [v]_{\mathrm{Lip}} \sum_{i=0}^{N-1} \int_0^T \rho_i(t) \int_{x_i(t)}^{x_{i+1}(t)} |\partial_x\varphi(x,t)|     \\
    \leq& (\lambda+ [V]_{\mathrm{Lip}((-\infty,0])}M(T)) m_N [v]_{\mathrm{Lip}}  T\|\partial_x \varphi\|_{L^1(\R \times (0,T))} \|\rho\|_{L^\infty(\R \times I_\varphi)},
\end{align*}
which vanishes as $m_N \to 0$.

\smallskip
\textbf{Step 4 - Initial condition.} 
The uniform control on the support of $\rho^N$ and the continuity in time of the particle trajectories easily imply that the curve $[0,T]\ni t\mapsto \rho^N(\cdot,t)$ is equicontinuous w.r.t. $N$ in the $p$-Wasserstein distance as $t\searrow 0$ for finite $p$. Hence, one can easily extract a subsequence of $\rho^N$ with a limit in $C([0,\delta];\Prob_q(\R))$ for some $q>1$. Such a limit, by construction, can only coincide with the limit as $N\rightarrow+\infty$ of $\rho^N(\cdot,0)$ at $t=0$, which is the initial condition $\rho_0$.
\end{proof}

\appendix
\section{Some technical results}\label{app}
We recall the following classical result.
\begin{thm}[Cauchy-Lipschitz-Picard \cite{brezis_FA}]\label{thm:CLP}
    Let $p\in [1,+\infty]$ and let $A:L^p([0,1])\rightarrow L^p([0,1])$ be a Lipschitz continuous operator. Let $\overline{X}\in L^p([0,1])$. Then, there exists one and only one $C^1$ curve $X(\cdot):[0,+\infty)\rightarrow L^p([0,1])$ such that
    \begin{equation}\label{eq:ODE_L2}
        \partial_t X(t) = A[X(t)]\qquad \hbox{for almost every $t\geq 0$}
    \end{equation}
    and $X(0)=\overline{X}$.
\end{thm}

Next, we provide the proof of Theorem \ref{thm:primitive}.
\begin{proof}[Proof of Theorem \ref{thm:primitive}]
    For a given test function $\psi(x,t)$, $\psi\in C^2_c(\R\times (0,+\infty))$, we have
    \[\int_0^{+\infty}\int_\R F(x,t)\left(\psi_t(x,t)+(G(x,t)\psi(x,t))_x\right)dxdt = 0\,.\]
 Now, let $\varphi\in C_c^1(\R\times (0,+\infty))$ and set $\psi(x,t)=\varphi_x(x,t)$. Defining the measure $\rho(\cdot,t)=\partial_x F(\cdot,t)$ and integrating by parts we get
 \[\int_0^{+\infty}\int_\R \left(\varphi_t(x,t)+G(x,t)\varphi_x(x,t)\right)d\rho(\cdot,t)(x) dt = 0\,,\]
 which proves $\rho$ satisfies the weak measure formulation of \eqref{eq:CE} for positive times.
\end{proof}

\section*{Acknowledgements}
The authors acknowledge useful discussions on this topic with Felisia A. Chiarello. The three authors are partially supported by the Italian “National Centre for HPC, Big Data and Quantum Computing” - Spoke 5 “Environment and Natural Disasters”. The research of MDF and SF is supported by the InterMaths Network, \url{www.intermaths.eu}, and the Ministry of University and Research (MIUR), Italy under the grant PRIN 2020- Project N. 20204NT8W4, Nonlinear Evolutions PDEs, fluid
dynamics and transport equations: theoretical foundations and applications.
ER and SF are supported by the INdAM project N.E53C22001930001 ``MMEAN-FIELDSS''. ER is also supported by  University of L’Aquila 2024 project 
"Deterministic particle schemes for crowd dynamics on networks and multidimensional domains" (04ATE2024.RIC.RADICI).

\bibliographystyle{plain}

\end{document}